\documentclass[11pt]{amsart}

\usepackage{amsmath,amsfonts,amsthm,amsopn,amssymb}
\usepackage{verbatim,wasysym,mathrsfs}
\usepackage[left=2.6cm,right=2.6cm,top=3.1cm,bottom=3.1cm]{geometry}
\usepackage{cite,marginnote}
\usepackage{color,enumitem,graphicx}
\usepackage[colorlinks=true,urlcolor=blue, citecolor=red,linkcolor=blue,
linktocpage,pdfpagelabels, bookmarksnumbered,bookmarksopen]{hyperref}
\usepackage[english]{babel}
\usepackage[T1]{fontenc}
\usepackage{lmodern}

\usepackage[hyperpageref]{backref}

\numberwithin{equation}{section}

\makeindex
\newtheorem{theorem}{Theorem}[section]
\newtheorem{proposition}[theorem]{Proposition}
\newtheorem{lemma}[theorem]{Lemma}
\newtheorem{remark}[theorem]{Remark}
\newtheorem{example}[theorem]{Example}
\newtheorem{corollary}[theorem]{Corollary}
\newtheorem{definition}[theorem]{Definition}

\newcommand{\RN}{\mathbb R^N}
\newcommand{\om}{\Omega}

\renewcommand{\O}{{\mathcal O}}
\newcommand{\iy}{\infty}

\newcommand{\g}{\gamma}
\newcommand{\G}{\Gamma}

\newcommand{\la}{\lambda}

\newcommand{\R}{\mathbb R}
\newcommand{\al}{\alpha}

\newcommand{\ti}{\tilde}

\newcommand{\re}[1]{(\ref{#1})}

\newcommand{\rg}{\rightarrow}

\newcommand{\e}{\varepsilon}
\newcommand{\vp}{\varphi}

\newcommand{\lab}{\label}
\newcommand{\bt}{\begin{theorem}}
\newcommand{\et}{\end{theorem}}
\newcommand{\bl}{\begin{lemma}}
\newcommand{\el}{\end{lemma}}
\newcommand{\bd}{\begin{definition}}
\newcommand{\ed}{\end{definition}}
\newcommand{\bc}{\begin{corollary}}
\newcommand{\ec}{\end{corollary}}
\newcommand{\bp}{\begin{proof}}
\newcommand{\ep}{\end{proof}}
\newcommand{\bx}{\begin{example}}
\newcommand{\ex}{\end{example}}
\newcommand{\bi}{\begin{exercise}}
\newcommand{\ei}{\end{exercise}}
\newcommand{\bo}{\begin{proposition}}
\newcommand{\eo}{\end{proposition}}
\newcommand{\br}{\begin{remark}}
\newcommand{\er}{\end{remark}}
\newcommand{\be}{\begin{equation}}
\newcommand{\ee}{\end{equation}}
\newcommand{\ba}{\begin{align}}
\newcommand{\ea}{\end{align}}
\newcommand{\bn}{\begin{enumerate}}
\newcommand{\en}{\end{enumerate}}
\newcommand{\bg}{\begin{align*}}
\newcommand{\bcs}{\begin{cases}}
\newcommand{\ecs}{\end{cases}}

\def\R{\mathbb R}

\def\N{\mathbb N}

\def\qed{\hfill$\square$\smallskip}
\makeatletter
\def\@makefnmark{}
\makeatother

\renewcommand{\H}{{\mathcal H}}

\newcommand{\eps}{\varepsilon}

\newcommand{\bean}{\begin{eqnarray*}}
\newcommand{\eean}{\end{eqnarray*}}

\renewcommand{\triangle}{\Delta}
\renewcommand{\epsilon}{\varepsilon}

\title[Ground states for fractional
critical Kirchhoff equations]{Ground states for fractional
	Kirchhoff equations \\ with critical nonlinearity in low dimension}

\author[Z. S. Liu]{Zhisu Liu}
\author[M. Squassina]{Marco Squassina}
\author[J. J. Zhang]{Jianjun Zhang}

\address[Z. S. Liu]{\newline\indent School of Mathematics and Physics,
University of South China
\newline\indent
Hengyang, Hunan 421001, P.R. China}
\email{\href{mailto:liuzhisu183@sina.com}{liuzhisu183@sina.com}}

\address[M.\ Squassina]{\newline\indent Dipartimento di Matematica e Fisica,
 Universita Cattolica del Sacro Cuore,
\newline\indent
Via Musei 41, 25121 Brescia, Italy}
\email{\href{mailto:marco.squassina@dmf.unicatt.it}{marco.squassina@unicatt.it}}

\address[J. J.\ Zhang]{\newline\indent College of Mathematics and Statistics,
Chongqing Jiaotong University
\newline\indent
Chongqing 400074, PR China}
\email{\href{mailto:zhangjianjun09@tsinghua.org.cn}{zhangjianjun09@tsinghua.org.cn}}

\thanks{Z. Liu is supported by the NSFC(11626127).
M.\ Squassina is member of the Gruppo Nazionale
	per l'Analisi Matematica, la Probabilit\`a
	e le loro Applicazioni (GNAMPA) of the Istituto Nazionale di Alta Matematica (INdAM).
J. J.\ Zhang was partially supported by the Science Foundation of Chongqing Jiaotong University(15JDKJC-B033).}


\begin{document}

\begin{abstract}
We study the existence of ground states to a nonlinear fractional Kirchhoff equation with an external potential $V$.
Under suitable assumptions
on $V$, using the monotonicity trick and the profile decomposition, we prove the existence of ground states. In particular, the nonlinearity does not satisfy the Ambrosetti-Rabinowitz type condition or monotonicity assumptions.
\end{abstract}
\maketitle

	\maketitle

\begin{center}
\begin{minipage}{8.5cm}
\small
\tableofcontents
\end{minipage}
\end{center}

\section{Introduction and results}
\label{sec1}

\subsection{Overview}
In this paper we are concerned with the existence of positive ground state solutions to the following nonlinear
fractional Kirchhoff equation
\begin{equation}\lab{HS1}\tag{K}
\left\{
\begin{aligned}
    \left(a+b\int_{\R^N}|(-\triangle)^{\frac{\alpha}{2}}u|^2dx\right)(-\triangle)^\alpha {u}+V(x)u =f(u) & \quad\mbox{in}\,\,\R^N, \\
    u\in H^\alpha(\R^N),\quad  u>0 & \quad\mbox{in}\,\, \R^N,
  \end{aligned}
\right.
\end{equation}
where $a,b$ are positive constants, $\alpha\in (0,1)$ and $N>2\alpha$.
The operator $(-\triangle)^{\alpha}$ is the fractional Laplacian
defined as ${\mathscr F}^{-1}(|\xi|^{2\alpha}{\mathscr F}(u))$,  where ${\mathscr F}$ denotes the
Fourier transform on $\mathbb{R}^N$.
When $a=1$ and $b=0$, then (K) reduces to the following fractional Schr\"{o}dinger equation
\begin{equation}\label{eqn:02}
 (-\triangle)^\alpha u+V(x)u =f(u) \quad\mbox{in}\,\,\R^N,
\end{equation}
which has been  proposed by Laskin \cite{Laskin02} in fractional quantum mechanics
as a result of extending the Feynman integrals from the Brownian like to the L\'{e}vy like quantum mechanical paths.
For such a class of fractional and nonlocal problems,
Caffarelli and Silvestre \cite{Caffarelli07} expressed
$(-\triangle)^\alpha$ as a Dirichlet-Neumann map for
a certain local elliptic boundary value problem on the half-space. This method is a valid tool to deal with equations
involving fractional operators to get regularity and handle variational methods. We refer the readers to
\cite{Silvestre06,Frank13} and to the references therein. Investigated first in  \cite{Felmer12,Dipierro13} via variational methods, there
has been a lot of interest in the study of the existence and multiplicity
of solutions for \eqref{eqn:02}
when $V$ and $f$ satisfy general conditions.
We cite \cite{Chang13,Servadei13--,Secchi13} with no attempts to provide a
complete list of references.

If $\alpha=1$, then problem (K) formally reduces to the well-known Kirchhoff equation
\begin{equation}\label{eqn:01}
 -\left(a+b\int_{\R^{N}}|\nabla u|^2dx\right)\triangle {u}+V(x)u =f(u) \quad\mbox{in}\,\,\R^N,
\end{equation}
related to the stationary analogue of the Kirchhoff-Schr\"{o}dinger type equation
$$
\frac{\partial^2 u}{\partial t^2}-\Big(a+b\int_{\Omega}|\nabla u|^2dx\Big)\triangle u=f(t,x,u),
$$
where $\Omega$ is a bounded domain in $\R^N$, $u$ denotes the displacement, $f$ is
the external force, $b$ is the initial tension and $a$ is related to the intrinsic
properties of the string. Equations of this type were
first proposed by Kirchhoff \cite{Kirchhoff83} in 1883 to describe the transversal oscillations
of a stretched string. Besides, we also
point out that such nonlocal problems appear in other fields like biological systems,
where $u$ describes a process depending on the average of itself. We refer readers to Chipot
and Lovat \cite{Chipot97}, Alves \emph{et al}. \cite{Alves01}. However,
the solvability of the Kirchhoff type equations has been
well studied in a general dimension by various authors only after J.-L.\ Lions \cite{Lions78} introduced an abstract
framework to such problems. For more recent results
concerning Kirchhoff-type equations we refer e.g.\ to \cite{Wu11,Azzollini11,
Ma03,Perera06,Zhang06,He11}.

In \cite{Li14}, by using a monotonicity trick and a global compactness lemma, Li and Ye proved that for $f(u)=|u|^{p-2}u$ and $p\in(3,2N/(N-2))$,
problem \eqref{eqn:01} has a positive ground state. Subsequently, Liu and Guo \cite{Liu15--} extended the above result to $p\in(2,2N/(N-2))$.
Fiscella and Valdinoci, in \cite{Fiscella14}, proposed the following stationary Kirchhoff
variational equation with critical growth
\begin{equation}\label{eqn:03}
\left\{
\begin{aligned}
    &M\left(\int_{\R^N}|(-\triangle)^{\frac{\alpha}{2}}u|^2dx\right)(-\triangle)^\alpha {u} =\lambda f(x,u)+|u|^{2_\alpha^*-2}u  \quad\mbox{in}\,\,\Omega, \\
    &u=0  \quad\mbox{in}\,\, \R^N\setminus\Omega,
  \end{aligned}
\right.
\end{equation}
which models nonlocal aspects of the tension arising from
measurements of the fractional length of the string.
They obtained the existence of non-negative solutions when $M$ and $f$ are
continuous functions satisfying suitable assumptions.
Autuori, Fiscella and Pucci \cite{Autuori15} considered the existence and the asymptotic
 behavior of non-negative solutions of \eqref{eqn:03}.\
Pucci and Saldi \cite{Pucci16-} established multiplicity of nontrivial solutions.
Via a three critical points
theorem, Nyamoradi \cite{Nyamoradi13} studied the subcritical case of \eqref{eqn:03} and obtained three solutions.
See also \cite{Pucci16,Xiang15} for related results.

To the best of our knowledge, there are few papers in the literature
on fractional Kirchhoff
equations in $\R^N$. Recently, Ambrosio and Isernia \cite{Ambrosio16} considered the fractional Kirchhoff problem
\begin{equation}\label{eqn:04}
\left(a+b\int_{\R^N}|(-\triangle)^{\frac{\alpha}{2}}u|^2dx\right)(-\triangle)^\alpha {u} =f(u)  \quad\mbox{in}\,\,\R^N, \\
\end{equation}
where $f$ is an odd subcritical nonlinearity
satisfying the well known Berestycki-Lions assumptions \cite{Berestycki1}. By minimax
arguments, the authors establish
a multiplicity result  in the radial space $H_{\rm rad}^\al(\R^N)$
when the parameter $b$ is {\em sufficiently small}. As in \cite{Li14}, Teng \cite{Teng16} also searched for ground state solutions for the
fractional Schr\"{o}dinger-Poisson system in $\R^3$ with critical growth
$$
\left\{
\begin{array}{ll}
(-\triangle)^{\alpha} {u}+ V(x)u+\phi u =\mu|u|^{q-2}u+|u|^{2^*_{\alpha}-2}u &\mbox{in}\,\,\R^3, \\
(-\triangle)^{t}{\phi}=u^2  &\mbox{in}\,\, \R^3.
\end{array}
\right.
$$
We point out that, in \cite{Li14,Teng16} the corresponding limit problems play an important r\v ole.
In order to get the existence of ground state solutions of the limit problems,  the authors used a
constrained minimization on a manifold $\mathcal{M}$ obtained by combining the Nehari and
Poho\v zaev manifolds.

\subsection{Main results}
Motivated by the works above, in this paper we aim to study the existence of
positive ground state solutions to the fractional Kirchhoff equation with
the Berestycki-Lions type conditions of {\em critical} type, firstly introduced in \cite{Zhang12}.

\subsubsection{Assumptions on $V$} On the external potential we assume the following:
\begin{itemize}
\item[\rm (V$_1$)] $V\in C^1(\R^N,\R)$ and, setting
$W(x):=\max\{x\cdot\nabla V(x),0\},$ we assume
$$
\|W\|_{L^{\frac{N}{2\alpha}}(\R^N)}<2a\alpha S_\alpha,
\qquad
S_{\alpha}:=\inf_{\overset{u\in D^{\alpha,2}(\R^N)}{u\neq 0}}\frac{\displaystyle\int_{\R^N}|(-\triangle)^{\frac{\alpha}{2}}u|^2dx}{
	\Big(\displaystyle\int_{\R^N}|u|^{2^*_\alpha}dx\Big)^{2/2^*_\alpha}},
\quad 2^*_\alpha:=\frac{2N}{N-2\alpha};
$$
\item[\rm (V$_2$)] there exists $V_\infty\in\R$ such that
$$
V(x)\leq\lim_{|y|\to\infty}V(y)=V_\infty,\quad
\text{for all $x\in\R^N$};
$$
\item[\rm (V$_3$)] the operator $a(-\triangle)^\alpha+V(x):H^\alpha(\R^N)\to H^{-\alpha}(\R^N)$ satisfies
$$
\inf_{\overset{u\in
	H^\alpha(\R^N)}{u\neq 0} }\frac{\displaystyle\int_{\R^N}\big(a|(-\triangle)^{\frac{\alpha}{2}}u|^2
	+V(x)u^2\big)dx}{\displaystyle\int_{\R^N}|u|^2dx}>0.
$$
\end{itemize}
\subsubsection{Assumptions on $f$}
We assume that $f(t)=0$ for all $t\leq 0$ and
\begin{itemize}
\item[\rm (f$_1$)]  $f\in C^1(\R^+,\R)$ and $\lim\limits_{t\rightarrow0}\frac{f(t)}{t}=0$;
\item[\rm (f$_2$)] $\lim\limits_{t\rightarrow\infty}\frac{f(t)}{t^{2_\alpha^*-1}}=1$;
\item[\rm (f$_3$)] there are $D > 0$ and $2<q<2_\alpha^*$ such that $f(t)\geq t^{2_\alpha^*-1}+Dt^{q-1}$ for any $t\geq 0$.
\end{itemize}

\vskip2pt
\noindent
Now we state our first result.
\begin{theorem}\label{Thm:ground1}
Assume
{\rm (V$_1$)-(V$_3$)}, {\rm (f$_1$)-(f$_3$)} and $N=2$ with $\alpha\in(\frac{1}{2},1)$ or $N=3$ with $\alpha\in(\frac{3}{4},1)$.
 \begin{itemize}
\item[\rm (i)]
If $q\in(2,2_\alpha^*)$, there is $D_1>0$ such that, for $D\geq D_1$, {\rm(K)} admits a positive ground state solution.
\item[\rm (ii)] If $q\in(\frac{4\alpha}{N-2\alpha}, 2_\alpha^*)$, for any $D>0$, {\rm(K)} admits a positive ground state solution.
\end{itemize}
\end{theorem}

We point out that without any symmetry assumption on $V$, the ground state solution obtained above maybe is not radially symmetric. In the following, we impose a monotonicity assumption of $V$ and show that {\rm(K)} admits a radially symmetric solution.

\vskip4pt
\noindent
 Assume now that $V$ is radially symmetric and increasing, that is 
\begin{equation}
\label{v4}
\tag{V$_4$}
\text{for all $x,y\in \R^N$:\,\,\ $|x|\leq |y|\,\,\Rightarrow\,\,\, V(x)\leq V(y)$}.
\end{equation}

\begin{theorem}\label{Thm:ground2}
Under the assumptions of Theorem \ref{Thm:ground1} and {\rm ($V_4$)}, 
{\rm(K)} admits a radially symmetric positive solution at the 
global (unrestricted to radial paths) mountain pass energy level.
\end{theorem}

As a main tool to prove Theorem~\ref{Thm:ground1}
we shall give the profile decomposition of the Palais-Smale sequences by which we can derive some compactness
and get a positive ground states for (K).
The main tool for the proof of Theorem~\ref{Thm:ground2} is a symmetric version
of the monotonicity trick \cite{symmontrick}.
We recall that Zhang and Zou \cite{Zhang14} studied the critical case for Berestycki-Lions theorem of the Schr\"{o}dinger equation $-\triangle u+V(x)u=f(u)$.
They obtained positive ground state solutions when $V$ satisfies similar assumptions
as (V$_1$)-(V$_3$), $f$ satisfies (f$_1$)-(f$_3$) and
\begin{itemize}
\item[\rm (f$_4$)] $|f'(t)|\leq C(1+|t|^{\frac{4}{N-2}}),\,\,\,
\text{for}\, t>0 \,\text{ and \,some }\, C>0.$
\end{itemize}
We should mention that in the present paper, (f$_4$) is removed.

\subsection{Main difficulties}
We mention the difficulties and the idea in proving Theorem \ref{Thm:ground1}.

Firstly, without the Ambrosetti--Rabinowitz condition, it is difficult to get the {\em boundedness of Palais-Smale sequences}.
In order to overcome this difficulty, inspired by \cite{Li14},
we will use the monotonicity trick developed by Jeanjean \cite{Jeanjean99},  introduce a family of functionals
$I_\lambda$ and obtain a bounded (PS)$_{c_\lambda}$ sequence for $I_\lambda$ for almost all $\lambda$ in an interval $J$,
where $c_\lambda$ is given in Section~\ref{la3}.

Secondly, by the presence of the Kirchhoff term, one obstacle arises in getting the compactness of $I_\lambda$, even in the subcritical case. Precisely,
this does not hold in general: for any $\vp\in C_0^\iy(\RN)$,
$$
\int_{\R^N}|(-\triangle)^{\frac{\alpha}{2}}u_n|^2dx
\int_{\RN}(-\triangle)^{\frac{\alpha}{2}}u_n(-\triangle)^{\frac{\alpha}{2}}\vp dx\rightarrow\int_{\R^N}|(-\triangle)^{\frac{\alpha}{2}}u|^2dx\int_{\RN}(-\triangle)^{\frac{\alpha}{2}}u(-\triangle)^{\frac{\alpha}{2}}\vp dx,
$$
where $\{u_n\}_{n\in\N}$ is a (PS)-sequence of $I_\la$ satisfying $u_n\rightharpoonup u$ in $H^\alpha(\R^N)$. Then, even in the subcritical case, it is not clear that weak limits are critical points of $I_\la$. In \cite{Ambrosio16}, for \eqref{eqn:04} the compactness was recovered by restricting $I_\lambda$
to the radial space $H_{\rm rad}^\alpha(\R^N)$,
which is compactly embedded $L^s(\R^N)$ for all $s\in(2,2_\alpha^*)$. For the related works in the bounded domains, see e.g.\ \cite{Fiscella14,Pucci16-,Xiang15}.

In the present paper, we do not impose any symmetry and just consider (K) in $H^\alpha(\R^N)$.
So the arguments mentioned above cannot be applied. Inspired by \cite{Li14}, in place of $I_\lambda$,
we consider a family of related functionals $J_\lambda$, whose corresponding problem is a
{\it non Kirchhoff} equation.

Thirdly, the critical exponent makes the problem rather tough. The (PS)-condition does {\em not} hold  in general and
to overcome this difficulty, we show that
the mountain pass level $c_\lambda$ is strictly less than some critical level $c_\lambda^*$.
For $-\triangle u+V(x)u=\lambda f(u)$ with critical growth,
if $S$ is the best constant of $D^{1,2}(\R^N)\hookrightarrow L^{2^{\ast}}(\R^N)$,
one can show that \cite {Brezis83-}
$$
c_\lambda^*=\frac{1}{N}S^{\frac{N}{2}}\lambda^{\frac{2-N}{2}}.
$$
For $-\left(a+b\int_{\R^{3}}|\nabla u|^2dx\right)\triangle {u}+V(x)u =\lambda f(u)$ in $\R^3$ involving critical growth \cite{Li13,Liu15}
$$
c^*_\lambda=\frac{ab}{4\lambda}S^3+\frac{[b^2S^4+4\lambda aS]^{\frac{3}{2}}}{24\lambda^2}+\frac{b^3S^6}{24\lambda^2}.
$$
However, for fractional Kirchhoff equations,
to give the {\em exact value} of $c_\lambda^*$ is complicated, since one cannot
solve precisely a fractional order algebra equation. A careful analysis is needed at this stage.
With an estimate of $c_\lambda$, inspired by \cite{Zhang14}, 
we establish a profile decomposition of the Palais-Smale sequence $\{u_n\}_{n\in\N}$ (Lemma \ref{Lem:global})
related to $J_\lambda$. Thanks to this result, for almost every $\la\in[1/2,1]$ we obtain a nontrivial critical point $u_\lambda$ of $I_\lambda$ at the level $c_\lambda$. 
Finally, choosing a sequence ${\lambda_n}\subset[1/2,1]$ with $\lambda_n\rightarrow1$,  
thanks to the Poho\v zaev identity we obtain
a bounded $(PS)_{c_1}$-sequence of the original functional $I$. Using the idea above again, we obtain a nontrivial solution of problem (K).

\vskip2pt
Throughout this paper, $C$ will denote a generic positive
constant.

The paper is organized as follows.

In Section 2, the variational setting and some preliminary
lemmas are presented.

In Section 3, we consider a perturbation of the original problem (K). Then using the monotonicity trick developed by Jeanjean, we obtain the bounded (PS)$_{c_\lambda}$-sequence $\{u_n\}_{n\in\N}$ for almost all $\lambda$.
In Section 4, an upper estimate of the mountain pass value is obtained and the limit problem is discussed. In Section 5,  we give the profile decomposition of $\{u_n\}_{n\in\N}$. In Section 6, Theorem 1.1 and 1.2 are finally proved.

\section{Variational setting}\setcounter{equation}{0}
\label{sec2} In this section we outline the variational framework
for (K) and recall some preliminary lemmas.
For any $\alpha\in(0,1)$, the fractional Sobolev space $H^{\alpha}(\R^3)$ is defined by
$$
H^{\alpha}(\R^N):=\left\{u\in L^2(\R^N):\,\frac{|u(x)-u(y)|}{|x-y|^{\frac{N+2\alpha}{2}}}\in L^2(\R^N\times\R^N)\right\}.
$$
It is known that
$$
\int_{\R^{2N}}\frac{|u(x)-u(y)|^2}{|x-y|^{N+2\alpha}}dxdy=2C(n,\alpha)^{-1}\int_{\R^N}|(-\triangle)^{\frac{\alpha}{2}}u|^2dx,
$$
where
$$
C(n,\alpha)=\left(\int_{\R^N}\frac{1-\cos{\zeta_1}}{|\zeta|^{N+2\alpha}}d\zeta\right)^{-1}.
$$
We endow the space $H^{\alpha}(\R^N)$ with the norm
$$
\|u\|_{H^\alpha(\R^N)}:=\left(\int_{\R^N}|u|^2dx+\int_{\R^N}|(-\triangle)^{\frac{\alpha}{2}}u|^2dx\right)^{1/2}.
$$
$H^\alpha(\R^N)$ is also the completion of $C^\infty_0(\R^N)$ with $\|\cdot\|_{H^\alpha(\R^N)}$ and it is
continuously embedded into $L^q(\R^N)$ for $q\in[2,2_\alpha^*]$.
The homogeneous space $D^{\alpha,2}(\R^N)$ is
$$
D^{\alpha,2}(\R^N):=\left\{u\in L^{2^*_{\alpha}}(\R^N):\,\frac{|u(x)-u(y)|}{|x-y|^{\frac{N+2\alpha}{2}}}\in L^2(\R^N\times\R^N)\right\},
$$
and it is also the completion of $C^\infty_0(\R^N)$ with respect to the norm
$$
\|u\|_{D^{\alpha,2}(\R^N)}:=\left(\int_{\R^N}|(-\triangle)^{\frac{\alpha}{2}}u|^2dx\right)^{1/2}.
$$

\begin{lemma}
\label{normeq}
Assume that {\rm (V$_2$)-(V$_3$)} hold. Then, for every $\eps>0$ there exists $\omega_\eps>0$ such that
$$
\int_{\R^N}\big((a-\eps)|(-\triangle)^{\frac{\alpha}{2}}u|^2+V(x)u^2\big)dx\geq \omega_\eps\int_{\R^N}|u|^2dx,
$$
for every $u \in H^\alpha(\R^N)$.
\end{lemma}
\begin{proof}
By contradiction, let $\{\eps_n\}_{n\in{\mathbb N}}\subset\R^+$ with $\eps_n\to 0$ and $\{u_n\}_{n\in {\mathbb N}}\subset H^\alpha(\R^N)$ with
$$
\int_{\R^N}\big((a-\eps_n)|(-\triangle)^{\frac{\alpha}{2}}u_n|^2+V(x)u^2_n\big)dx\leq \frac{1}{n}\int_{\R^N}|u_n|^2dx.
$$
Then, up to a standard nomalization, we may assume that $\|u_n\|_{H^\alpha(\R^N)}=1$ and
$$
\int_{\R^N}\big((a-\eps_n)|(-\triangle)^{\frac{\alpha}{2}}u_n|^2+V(x)u^2_n\big)dx\leq \frac{1}{n}.
$$
In view of (V$_3$), we get $\|u_n\|_2\to 0$, which implies from (V$_2$) and the above inequality
that $\{u_n\}_{n\in\N}$ goes to zero in $D^{\alpha,2}(\R^N)$. Therefore $u_n\to 0$ in $H^\alpha(\R^N)$, which is a contradiction.
\end{proof}

\noindent
Let
$$
\H:=\left\{u\in H^{\alpha}(\R^N):\,\int_{\R^N}V(x)u^2dx<\infty\right\}
$$
be the Hilbert space equipped with the inner product
$$
\langle u,v\rangle_{\H} :=a\int_{\R^N} (-\Delta)^{\frac{\alpha}{2}}u (-\Delta)^{\frac{\alpha}{2}}v \,dx+\int_{\R^N}V(x)uv \,dx,
$$
and the corresponding induced norm
 $$
 \|u\|:=\left(\int_{\R^N}a|(-\triangle)^{\frac{\alpha}{2}}u|^2dx+\int_{\R^N}V(x)u^2dx\right)^{1/2}.
 $$
From Lemma~\ref{normeq}, it easily follows that the above norm is equivalent to $\|\cdot\|_{H^\alpha}$.
A function $u\in\H$ is a (weak) solution to problem \eqref{HS1} if, for every $\varphi\in\H$, we have
\begin{equation*}
\left(a+b\int_{\R^N}|(-\triangle)^\frac{\alpha}{2} u|^2dx\right)\int_{\R^N}(-\triangle)^{\alpha/2} u (-\triangle)^{\alpha/2}\varphi dx
+\int_{\R^N}V(x)u\varphi dx=\int_{\R^N} f(u)\varphi dx.
\end{equation*}
We stress that, under assumptions {\rm (V$_1$)-(V$_3$)} and {\rm (f$_1$)-(f$_3$)}, if $u$ is a weak solution to the above problem, then $u$ is globally bounded and H\"older regular allowing the pointwise reppresentation of $(-\Delta)^\alpha u$ by the results of \cite{Felmer12}. In particular $u>0$ a.e.\ wherever $u\geq 0$.


\begin{lemma}[Lions lemma, see \cite{Secchi13}]\label{Lem:P.lions1}
Assume that $\{u_n\}_{n\in\N}$ is bounded in $\H$ and
$$
\lim\limits_{n\rightarrow\infty}\sup\limits_{y\in\R^N}\int_{B_{r}(y)}|u_n|^2dx=0,
$$
for some $r>0$. Then $u_n\rightarrow0$ in $L^s(\R^N)$ for all $s\in(2,2_\alpha^*)$.
\end{lemma}

\noindent
The energy functional associated with (K), $I:\H\to\R$, is defined as
\begin{equation*}
I(u)=\frac{1}{2}\|u\|^2+\frac{b}{4}\left(\int_{\R^N}|(-\triangle)^{\frac{\alpha}{2}}u|^2dx\right)^2-\int_{\R^N}F(u)dx,\quad u\in\H,
\end{equation*}
with $F(u)=\int_{0}^{u}f(t)dt$. Obviously $I\in C^1(\H)$ and its critical points are weak solutions to (K).


\section{The perturbed functional}
\label{la3}
Since we do not impose the well-known {\it Ambrosetti-Rabinowitz} condition, the boundedness of the Palais-Smale sequence becomes complicated. To overcome this difficulty, we adopt a monotonicity trick due to Jeanjean \cite{Jeanjean99}.

\begin{theorem}[Monotonicity trick \cite{Jeanjean99}]
	\label{Thm:MP}
Let $(E,\|\cdot\|)$ be a real Banach space with its dual space
$E'$ and $J\in\R^+$ an interval. Consider the family of $C^1$
functionals on $E$
$$
I_{\lambda}=A(u)-\lambda B(u),\,\, \forall\lambda\in J,
$$
with $B$ nonnegative and either $A(u)\rightarrow+\infty$ or
$B(u)\rightarrow+\infty$ as $\|u\|\rightarrow\infty$, satisfying
$I_{\lambda}(0)=0.$ We set
$$
\Gamma_\lambda:=\{\gamma\in C([0,1],E)\,\, |\,\,\gamma(0)=0,\, I_\lambda(\gamma(1))<0\},
\quad\text{for all $\lambda\in J$}.
$$
If for every $\lambda\in J$,
$\Gamma_\lambda$ is nonempty and
\begin{equation*}
c_\lambda=
\inf\limits_{\gamma\in\Gamma_\lambda}\max\limits_{s\in[0,1]}I_\lambda{(\gamma(s))}>0,
\end{equation*}
then for almost any $\lambda\in J$, $I_\lambda$ admits a
bounded Palais-Smale sequence $\{u_n\}_{n\in\N}\subset E$,
namely $\sup_{n\in\N}\|u_n\|<\infty$, $I_{\lambda}(u_n)\rightarrow c_\lambda$ and
$I'_\lambda(u_n)\rightarrow 0$ in $E'$. Moreover
$\lambda\rightarrow c_\lambda$ is left continuous.
\end{theorem}

\noindent
Set $J:=[\frac{1}{2},1]$, $E:=\H$ and
$$
\aligned
&A(u):=\frac{1}{2}\|u\|^2+\frac{b}{4}\left(\int_{\R^N}|(-\triangle)^{\frac{\alpha}{2}}u|^2dx\right)^2,\quad B(u):=\int_{\R^N}F(u)dx.
\endaligned
$$
We consider the family of functionals
$I_\lambda: \H\rightarrow\R$ defined by $I_\lambda(u)=A(u)-\lambda B(u)$, that is
\begin{equation*}
I_{\lambda}(u)=
\frac{1}{2}\|u\|^2+\frac{b}{4}\left(\int_{\R^N}|(-\triangle)^{\frac{\alpha}{2}}u|^2dx\right)^2-
\lambda\int_{\R^N}F(u)dx.
\end{equation*}
It is easy to see that $B(u)\geq 0$ for all $u\in \H$ and $A(u)\rightarrow+\infty$ as
$\|u\|\rightarrow\infty$.

\noindent
In the following $H$ denotes a closed half-space of $\R^N$ containing the origin, $0\in H$.
We denote by ${\mathscr H}$ the set of closed half-spaces of $\R^N$ containing the origin. We shall
equip ${\mathscr H}$ with a topology ensuring that $H_n\to H$ as $n\to\infty$ if there is a sequence of
isometries $i_n:\R^N\to\R^N$ such that $H_n=i_n(H)$ and $i_n$ converges to the identity.
Given $x\in\R^N$, the reflected point $\sigma_H(x)$ will also be denoted by $x^H$.
The polarization of a nonnegative function
$u:\R^N\to\R_+$ with respect to $H$ is defined as
$$
u^H(x):=
\begin{cases}
\max\{u(x),u(\sigma_H(x))\}, & \text{for $x\in H$}, \\
\min\{u(x),u(\sigma_H(x))\}, & \text{for $x\in \R^N\setminus H$}.
\end{cases}
$$
Given $u,$ the Schwarz symmetrization $u^*$ of $u$ is the unique
function such that $u$ and $u^*$ are equimeasurable and $u^*(x)=h(|x|)$, where $h:(0,\infty)\to\R_+$ is a continuous and decreasing function.

\vskip3pt
\noindent
We set $\H^+:=\{u\in\H:u\geq 0\}.$ Now we state a symmetric version of  Theorem~\ref{Thm:MP}.

\begin{lemma}[Symmetric monotonicity trick \cite{symmontrick}]
\label{symm-vers}
Under the assumptions of Theorem~\ref{Thm:MP} for $E=\H$, assume that
$I_\lambda(|u|)\leq I_\lambda(u)$ for any $\lambda\in J$ and  $u\in \H$ and
$$
I_\lambda(u^H)\leq I_\lambda(u),\quad  \text{for any $\lambda\in J$, $u\in \H^+$  and $H\in {\mathscr H}$}.
$$
Then, for any $p\in[2,2^*_\alpha]$, $I_\lambda$ has a bounded Palais-Smale sequence $\{u_n\}_{n\in\N}\subset \H$ with $\|u_n-|u_n|^*\|_p\to 0$.
\end{lemma}

\begin{lemma}[Uniform Mountain-Pass geometry]
\label{Lem:MP1}
Assume that {\rm (f$_1$)}-{\rm(f$_3$)} and {\rm (V$_1$)-\rm(V$_3$)} hold. Furthermore
let $N=2$ with $\alpha\in(\frac{1}{2},1)$ or $N=3$ with $\alpha\in(\frac{3}{4},1)$.
Then we have: \\
{\rm(1)} $\Gamma_\lambda\neq\emptyset$, for every $\lambda\in J$;\\
{\rm(2)}  there exist $r,\eta>0$ independent of $\lambda$, such that $\|u\|=r$ implies $I_\lambda(u)\geq \eta$. In particular $c_\lambda\geq\eta$.
\end{lemma}
\begin{proof} (1)
For every $\varphi\in \H^+\setminus\{0\}$, taking into account of (f$_3$), we have
$$
I_\lambda(\varphi)\leq\frac{1}{2}\|\varphi\|^2+\frac{b}{4}\left(\int_{\R^N}|(-\triangle)^{\frac{\alpha}{2}}\varphi|^2dx\right)^2-
\frac{D}{2q}\int_{\R^N}\varphi^qdx
-\frac{1}{22_\alpha^*}\int_{\R^N}\varphi^{2_\alpha^*}dx.
$$
Under the assumptions on $N$ and $\alpha$, it follows that $2^*_\alpha>4$.
Then there exists $t_0>0$ sufficiently large, independent of $\lambda\in J$, such that $I_\lambda(t_0\varphi)<0$. Setting $w:=t_0\varphi\in \H$, we have $I_\lambda(w)<0$ and we can define the corresponding $\Gamma_\lambda$. Then, setting $\gamma(t):=t w$, we have $\gamma\in \Gamma_\lambda\neq\emptyset$ for every $\lambda\in J$.

(ii) (f$_1$)-(f$_2$) imply that, for any
$\epsilon>0$, there exists $C_\epsilon>0$ such that
\begin{equation*}
|F(s)|\leq \epsilon|s|^2+C_\epsilon|s|^{2_\alpha^*},\quad\text{for all $s\in\R$}.
\end{equation*}
Then there exist $\sigma_1,\sigma_2>0$ such that
$$
I_\lambda(\varphi)\geq \sigma_1\|\varphi\|^2-\sigma_2\|\varphi\|^{2_\alpha^*},
\quad\text{for every $\varphi\in \H$}.
$$
Hence there exist $r,\eta>0$, independent of $\lambda$, such that for $\|u\|=r$, $I_\lambda(u)\geq \eta$ (and $I_\lambda(\varphi)>0$
as soon as $\|\varphi\|\leq r$ with $\varphi\neq 0$). Now fix
$\lambda\in J$ and $\gamma\in\Gamma_\lambda$. Since $\gamma(0)=0$
and $I_{\lambda}(\gamma(1))<0$, certainly $\|\gamma(1)\|>r$. By
continuity, we conclude that there exists $t_\gamma\in(0,1)$ such
that $\|\gamma(t_\gamma)\|=r$. Therefore, for every $\lambda\in
J$, we conclude $c_{\lambda}\geq
\inf_{\gamma\in\Gamma_\lambda}I_\lambda(\gamma(t_{\gamma}))\geq
\eta>0.$
\end{proof}

\begin{lemma}[$I_\lambda$ decreases upon polarization]
\label{decr}
Assume {\rm(V$_4$)} holds. Then
for any $\lambda\in J$, for all $u\in \H^+$
and $H\in {\mathscr H}$ there holds $I_\lambda(u^H)\leq I_\lambda(u)$.
\end{lemma}
\begin{proof}
It is known (see \cite[Theorem 2]{baer}) that
$$
\int_{\R^{2N}}\frac{|u^H(x)-u^H(y)|^2}{|x-y|^{N+2\alpha}}dxdy\leq \int_{\R^{2N}}\frac{|u(x)-u(y)|^2}{|x-y|^{N+2\alpha}}dxdy,\quad\text{for all $u\in \H^+$.}
$$
Furthermore, we have (see \cite{jean})
$$
\int_{\R^N} F(u^H)dx=\int_{\R^N} F(u)dx,\quad\text{for all $u\in \H^+$},
$$
and, by the monotonicity assumptions on $V$,
$$
\int_{\R^N} V(x)(u^H)^2dx\leq \int_{\R^N} V(x)u^2dx,\quad\text{for all $u\in \H^+$},
$$
which concludes the proof by the definition of $I_\lambda$.
\end{proof}

\noindent Assume {\rm (V$_1$)-(V$_3$)} and {\rm (f$_1$)-(f$_3$)}. As a consequence we now get the following result.

\begin{corollary}[Bounded Palais-Smale with sign]
\lab{bps}
For almost every $\lambda\in J$, there is a bounded sequence
$\{u_n\}_{n\in \N}\subset \H^+$ such that $I_\lambda(u_n)\rightarrow c_\lambda$, $I'_\lambda(u_n)\to 0$.
Furthermore, $\|u_n-|u_n|^*\|_{2^*_\alpha}\to 0$ if \eqref{v4} is assumed.
\end{corollary}
\begin{proof}
For a.a.\ $\lambda\in J$,
a bounded (PS)-sequence $\{u_n\}_{n\in\N}\subset \H$ for $I_\lambda$
is provided by combining Theorem~\ref{Thm:MP} with Lemmas \ref{Lem:MP1} and \ref{decr}. Furthermore,  if \eqref{v4} holds,
using Lemma \ref{symm-vers} in place of Theorem~\ref{Thm:MP}, we also get $\|u_n-|u_n|^*\|_{2^*_\alpha}\to 0$. Next we show that
we can assume that $u_n$ is nonnegative. Indeed, we
know that $\langle I'_\lambda(u_n), u_n^-\rangle=\langle\mu_n,u^-_n\rangle$ with $\mu_n\to 0$ in $\H'$ as $n\to\infty$, with
$u_n^- = \min\{u_n,0\}$, namely ($f(s)=0$ for $s\leq 0$)
\begin{equation*}
\left(a+b\int_{\R^N}|(-\triangle)^\frac{\alpha}{2} u_n|^2dx\right)\int_{\R^N}(-\triangle)^\alpha u_nu_n^-dx
+\int_{\R^N}V(x)|u_n^-|^2dx=\langle\mu_n,u^-_n\rangle.
\end{equation*}
As it is readily checked, for all $x,y\in\R^N$, we have
$$
(u_n(x)-u_n(y))(u_n^-(x)-u_n^-(y)) \geq (u_n^-(x)-u_n^-(y))^2,
$$
which yields that
\begin{align*}
2C(n,\alpha)^{-1}\int_{\R^3}(-\triangle)^\alpha u_nu_n^-dx &=\int_{\R^{2N}}\frac{[u_n(x)-u_n(y)][u_n^-(x)-u_n^-(y)]}{|x-y|^{N+2\alpha}}dxdy\\
&\ge\int_{\R^{2N}}\frac{[u_n^-(x)-u_n^-(y)]^2}{|x-y|^{N+2\alpha}}dxdy=2C(n,\alpha)^{-1}\|u_n^-\|_{D^{\alpha,2}}^2.
\end{align*}
Thus $\|u_n^-\|=o_n(1)$, which also yields that $\{ u^+_n\}_{n\in\N}$
is bounded. We can now prove that $I_\lambda(u^+_n )\rightarrow c_\lambda$ and $I'_\lambda(u^+_n)\rightarrow0$.
Of course $\|u_n\|^2=\|u_n^+\|^2+o_n(1)$ and
$$
\left(\int_{\R^N}|(-\triangle)^{\frac{\alpha}{2}}u_n|^2dx\right)^2=\left(\int_{\R^N}|(-\triangle)^{\frac{\alpha}{2}}u_n^+|^2dx\right)^2+o_n(1).
$$
Notice that from (f$_1$)-(f$_2$), we get
$$
\left|\int_{\R^N} F(u_n)dx-\int_{\R^N} F(u_n^+)dx\right|
\leq C\int_{\R^N}(|u_n|+|u_n|^{2^*_\alpha-1})|u_n^-|\leq C\|u_n^-\|_2+C\|u_n^-\|_{2^*_\alpha}=o_n(1).
$$
This shows that $I_\lambda(u^+_n )\rightarrow c_\lambda$.
We claim that $I'_\lambda(u^+_n)\to 0$. Setting
$w_n:=I'_\lambda(u_n)-I'_\lambda(u^+_n)$, it is enough to prove that $w_n\to 0$ in $\H'$. For any $\varphi\in\H$ with $\|\varphi\|_\H\leq 1,$
we have
\begin{align*}
\langle w_n,\varphi\rangle &=\left(a+b\int_{\R^N}|(-\triangle)^\frac{\alpha}{2} u_n|^2dx\right)\int_{\R^N}(-\triangle)^{\alpha/2} u_n (-\triangle)^{\alpha/2}\varphi dx \\
&-\left(a+b\int_{\R^N}|(-\triangle)^\frac{\alpha}{2} u_n^+|^2dx\right)\int_{\R^N}(-\triangle)^{\alpha/2} u_n^+ (-\triangle)^{\alpha/2}\varphi dx\\
&+\int_{\R^N}V(x)u_n^-\varphi dx-\lambda\int_{\R^N} (f(u_n)-f(u_n^+))\varphi dx \\
&=\left(a+b\int_{\R^N}|(-\triangle)^\frac{\alpha}{2} u_n^+|^2dx\right)\int_{\R^N}(-\triangle)^{\alpha/2} u_n^- (-\triangle)^{\alpha/2}\varphi dx \\
&+\int_{\R^N}V(x)u_n^-\varphi dx-\lambda\int_{\R^N} f(u_n^-)\varphi dx +\langle \xi_n,\varphi\rangle,
\end{align*}
for some $\xi_n\to 0$ in $\H'$. Then, using (f$_1$)-(f$_2$) again, $|\langle w_n,\varphi\rangle |\leq C\|u^-_n\|_{\H}+\|\xi_n\|_{\H'},$
proving the claim. Observe now that by the triangular inequality and the contractility
property of the Schwarz symmetrization in $L^p$-spaces
(i.e.\ $\|w^*-z^*\|_p\leq \|w-z\|_p$ for all $w,z\in L^p(\R^N)$ with $w,z\geq 0$), we get
\begin{align*}
& \left| \|u_n^+-(u_n^+)^*\|_{2^*_\alpha}- \|u_n-|u_n|^*\|_{2^*_\alpha} \right| \\
&\leq \|u_n^-+((u_n^+)^*-|u_n|^*)\|_{2^*_\alpha}\leq \|u_n^-\|_{2^*_\alpha}+\|(u_n^+)^*-|u_n|^*\|_{2^*_\alpha} \\
&\leq \|u_n^-\|_{2^*_\alpha}+\|u_n^+-|u_n|\|_{2^*_\alpha}=2\|u_n^-\|_{2^*_\alpha}\leq C\|u_n^-\|_{\H}=o_n(1).
\end{align*}
Since $\|u_n-|u_n|^*\|_{2^*_\alpha}\to 0$, we have  $\|u_n^+-(u_n^+)^*\|_{2^*_\alpha}\to 0$ as $n\to\infty$.
This ends the proof.
\end{proof}




\section{Upper estimate of $c_\lambda$ and limit problems}

\noindent

In this section, we give an upper estimate of the mountain pass value $c_\lambda$. Moreover, the corresponding limit problem is discussed.

\subsection{An energy estimate}
Next we provide a crucial energy estimate for $c_\lambda$.

\begin{lemma}[Energy estimate]\label{Lem:level}
	Suppose that {\rm (f$_1$)}-{\rm(f$_3$)} and {\rm (V$_1$)-\rm(V$_3$)} hold. For any $\lambda\in[\frac{1}{2},1]$, assume that
	$$
	q\in(\frac{4\alpha}{N-2\alpha}, 2_\alpha^*)
	\quad
	\text{or}
	\quad
	\text{$q\in(2,\frac{4\alpha}{N-2\alpha}]$ with $D$ large enough}.
	$$
	Then we have
	$$
	c_\lambda<c_\lambda^*,\qquad
	c_\lambda^*:=\frac{a S_\alpha}{2}T^{N-2\alpha}+\frac{bS_\alpha^2}{4}T^{2N-4\alpha}-\frac{\lambda}{2_\alpha^*}T^{N},
	$$
	where $T=T(\lambda)>0$ continuously depends on $\lambda$.
\end{lemma}
\begin{proof}
			Let
			$\eta\in C_0^\infty(\R^3)$ be a cut-off function with support in $B_2(0)$ such that $\eta\equiv1$ on $B_1(0)$ and $\eta\in [0,1]$ on $B_2(0)$.
			It is well known that $S_\alpha$ is achieved by
			$$
			{\mathcal T}(x):=\kappa {\left(\mu^2+|x-x_0|^2\right)^{-\frac{N-2\alpha}{2}}}
			$$
			for arbitrary $\kappa\in\R$, $\mu>0$ and $x_0\in\R^N$. Then, taking $x_0=0$,
			we can define
			$$
			v_\varepsilon(x):=\eta(x)u_\e(x),  \quad			
			u_\e(x)=\varepsilon^{-\frac{N-2\alpha}{2}}u^*(x/\eps),
			 \quad u^*(x):=\frac{{\mathcal T}\big(x/S_\alpha^{1/(2\alpha)}\big)}{\|{\mathcal T}\|_{2^*_\alpha}}.
			$$
			Then $(-\Delta)^{\alpha}u_\e=|u_\e|^{2^*_\alpha-2}u_\e$ and $\|(-\triangle)^{\frac{\alpha}{2}}u_\e\|_2^2=\|u_\e\|_{2^*_\alpha}^{2^*_\alpha}=S_\al^{\frac{N}{2\al}}$. As in \cite{Servadei13--}, we have
			\begin{align}\label{eqn:3iu3}
				A_\e &:=\int_{\R^N}|(-\triangle)^{\frac{\alpha}{2}} v_\varepsilon(x)|^2dx= S_\al^{\frac{N}{2\al}}+\O(\varepsilon^{N-2\alpha}).
			\end{align}
			On the other hand, for any $q\in[2,2_\alpha^*)$, we obtain
			$$
			\int_{\R^N}|v_\varepsilon|^{q}dx\ge\varepsilon^{N-\frac{(N-2\al)q}{2}}\kappa^q\|{\mathcal T}\|_{2^*_\alpha}^{-q}|\mathbb{S}_{N-1}|S_\alpha^{N/(2\alpha)}\int_0^{\frac{1}{\e S_\al^{1/(2\al)}}}\frac{r^{N-1}}{(\mu^2+r^2)^{\frac{(N-2\al)q}{2}}}dr,
			$$
            where $\mathbb{S}_{N-1}$ is the unit sphere in $\R^N$.
		    Observe that, as $\e\rg0$,
			$$
			\int_0^{\frac{1}{\e S_\al^{1/(2\al)}}}\frac{r^{N-1}}{(\mu^2+r^2)^{\frac{(N-2\al)q}{2}}}dr
			\bcs
			\rg c\in(0,\iy),\,\,&\mbox{if}\,\,\,q>\frac{N}{N-2\al},\\
			=\O(\log{(\frac{1}{\e})}),\,\,&\mbox{if}\,\,\,q=\frac{N}{N-2\al},\\
			=\O(\e^{(N-2\al)q-N}),\,\,&\mbox{if}\,\,\,q<\frac{N}{N-2\al}.
			\ecs
			$$
			Then
			\begin{align}
				\label{eqn:3iu5}
				C_\e&:=\int_{\R^N}|v_\varepsilon|^{q}dx\ge
				\bcs
				\O(\varepsilon^{N-\frac{(N-2\al)q}{2}}),\,\,&\mbox{if}\,\,\,q>\frac{N}{N-2\al},\\
				\O(\log{(\frac{1}{\e})}\varepsilon^{N-\frac{(N-2\al)q}{2}}),\,\,&\mbox{if}\,\,\,q=\frac{N}{N-2\al},\\
				\O(\e^{\frac{(N-2\al)q}{2}}),\,\,&\mbox{if}\,\,\,q<\frac{N}{N-2\al}.
				\ecs
			\end{align}
			Since $2<\frac{N}{N-2\al}$,
			$$
			\int_{\R^N}|v_\varepsilon|^2dx\ge\O(\varepsilon^{N-2\alpha}).
			$$
			Similar as above,
			$$
			\int_{\R^N}|v_\varepsilon|^2dx\le\varepsilon^{2\al}\kappa^2\|{\mathcal T}\|_{2^*_\alpha}^{-2}|\mathbb{S}_{N-1}|S_\alpha^{N/(2\alpha)}\int_0^{\frac{2}{\e S_\al^{1/(2\al)}}}\frac{r^{N-1}}{(\mu^2+r^2)^{N-2\al}}dr\le\O(\varepsilon^{N-2\alpha}).
			$$
			So that we have
			\begin{align}\label{eqn:3iu4}
				B_\e&:=\int_{\R^N}|v_\varepsilon|^2dx=\O(\varepsilon^{N-2\alpha}).
			\end{align}
			As can be seen in \cite{Servadei13--}, it holds
			$$
			D_\e:=\int_{\RN}|v_\e|^{2^*_\alpha}dx=S_\al^{\frac{N}{2\al}}+\O(\e^N).
			$$
	\noindent
	{\bf Step 1.} For any $\e>0$ small there exists $t_0>0$ such that $I_\la(\g_\e(t_0))<0$, where $\g_\e(t):=v_\varepsilon(\cdot/t)$. Indeed, by {\rm (V$_2$)} and \rm (f$_3$), for any $t>0$,
	\begin{align}\lab{yinyong}
		 I_\la(\g_\e(t))&\le\frac{a}{2}\int_{\R^N}|(-\triangle)^{\frac{\alpha}{2}}\g_\e(t)|^2dx+\frac{b}{4}\left(\int_{\R^N}|(-\triangle)^{\frac{\alpha}{2}}\g_\e(t)|^2dx\right)^2\nonumber\\
		 &\,\,\,\,\,\,\,+\frac{V_\iy}{2}\int_{\R^N}|\g_\e(t)|^2dx-\lambda\int_{\RN}\left[\frac{|\g_\e(t)|^{2_\alpha^*}}{2_\alpha^*}+D\frac{|\g_\e(t)|^q}{q}\right]dx\nonumber\\
		&=\frac{aA_\e}{2}t^{N-2\alpha}+\frac{bA_\e^2}{4}t^{2N-4\alpha}+\left(\frac{V_\iy B_\e}{2}-\frac{\la D_\e}{2_\alpha^*}-\frac{\la DC_\e}{q}\right)t^{N}.
	\end{align}
	Noting that $2\al<N<4\al$, we have $0<2N-4\al<N$. Then by \re{eqn:3iu4},
	$$
		\frac{V_\iy B_\e}{2}-\frac{\la D_\e}{2_\alpha^*}\rg-\frac{\la S_\al^{\frac{N}{2\al}}}{2_\alpha^*},\,\,\mbox{as}\,\,\,\e\rg0.
		$$
	So it follows from \re{eqn:3iu3} that for any $\e>0$ small enough, $I_\la(\g_\e(t)\rg-\iy$ as $t\rg+\iy$. Then there exists $t_0>0$ such that $I_\la(\g_\e(t_0))<0$.
	\vskip0.1in
	\noindent
	{\bf Step 2.} Notice that, as $t\rg 0^+$, we have
	$$
	\int_{\R^N}\left[|(-\triangle)^{\frac{\alpha}{2}}\g_\e(t)|^2+|\g_\e(t)|^2\right]dx=t^{N-2\al}A_\e+t^NB_\e\rg0
	$$
	uniformly for $\e>0$ small. We set $\g_\e(0)=0$. Then $\g_\e(t_0\cdot)\in\G_\la$, where $\G_\la$ is as in Theorem \ref{Thm:MP} and  $$
	c_\lambda\leq\sup\limits_{t\ge0}I_\lambda(\g_\e(t)).
	$$
	Recalling that $c_\la>0$, by \re{yinyong}, there exists $t_\e>0$ such that
	$$
	\sup\limits_{t\ge0}I_\lambda(\g_\e(t))=I_\lambda(\g_\e(t_\e)).
	$$
	By \re{eqn:3iu3}, \re{eqn:3iu5} and \re{yinyong}, we get $I_\lambda(\g_\e(t))\rg 0^+$ as $t\rg0^+$ and $I_\lambda(\g_\e(t))\rg-\iy$ as $t\rg+\iy$ uniformly for $\e>0$ small. Then there exist $t_1,t_2>0$ (independent of $\e>0$) such that $t_1\le t_\e\le t_2$.
	Let
	$$
	J_\varepsilon(t):=\frac{aA_\e}{2}t^{N-2\alpha}+\frac{bA_\e^2}{4}t^{2N-4\alpha}-\frac{\la D_\e}{2_\alpha^*}t^{N},
	$$
	then
	$$
	c_\la\le\sup_{t\ge0}J_\varepsilon(t)+\left(\frac{V_\iy B_\e}{2}-\frac{\la DC_\e}{q}\right)t_\e^{N}
	$$
	By formula \re{eqn:3iu5}, for any $q\in(2,2_\alpha^*)$, we have
		$$
		C_\e\ge\O(\e^{N-\frac{(N-2\al)q}{2}}).
		$$
	Then by \re{eqn:3iu4}, we conclude that
	$$
	c_\la\le\sup_{t\ge0}J_\varepsilon(t)+\O(\varepsilon^{N-2\alpha})-\O(D\varepsilon^{N-\frac{(N-2\alpha)q}{2}}).
	$$
	Noting that $N-2\alpha>0$ and $N-(N-2\alpha)q/2>0$, we have $\sup_{t\ge0}J_\varepsilon(t)\ge c_\la/2$ uniformly for $\e>0$ small. As above, there are $t_3,t_4>0$ (independent of $\e>0$) such that $\sup_{t\ge0}J_\varepsilon(t)=\sup_{t\in[t_3,t_4]}J_\varepsilon(t)$. By \re{eqn:3iu3},
	\be\lab{yinyong2}
		c_\la\le\sup_{t\ge0}K(S_\al^{\frac{1}{2\al}}t)+\O(\varepsilon^{N-2\alpha})-\O(D\varepsilon^{N-\frac{(N-2\alpha)q}{2}}),
		\ee
	where
	$$K(t)=\frac{aS_\al}{2}t^{N-2\alpha}+\frac{bS_\al^2}{4}t^{2N-4\alpha}-\frac{\la}{2_\alpha^*}t^{N}.$$
	Observe that for $t>0$,
	$$
	K'(t)=\frac{(N-2\alpha)t^{N-2\alpha-1}}{2}\ti{K}(t),\,\,\mbox{where}\,\,\,\ti{K}(t):=aS_\al+bS_\al^2t^{N-2\alpha}-\lambda t^{2\al},
	$$
	and $\ti{K}'(t)=t^{N-2\alpha-1}(bS_\al^2(N-2\alpha)-2\lambda\al t^{4\al-N})$. Since $4\alpha>N$, there is a unique $T>0$ such that $\ti{K}(t)>0$ if $t\in(0,T)$ and $\ti{K}(t)<0$ if $t>T$. Hence, $T$ is the unique maximum point of $K$. Then by \re{yinyong2},
	\be\lab{yinyong3}
	c_\la\le K(T)+\O(\varepsilon^{N-2\alpha})-\O(D\varepsilon^{N-\frac{(N-2\alpha)q}{2}}).
	\ee
	If $q>4\alpha/(N-2\alpha)$, then $0<N-(N-2\alpha)q/2<N-2\alpha$, which implies by \re{yinyong3} that for any fixed $D>0$, $c_\la<K(T)$ for $\e>0$ small. If $2<q\le 4\alpha/(N-2\alpha)$, for $\e>0$ small and
	$D\ge\e^{(N-2\al)q/2-2\al-1}$, then also in this case $c_\la<K(T),$
which completes the proof.
\end{proof}

\subsection{The limit problem}
Note that $V(x)\rightarrow V_{\infty}$ as $|x|\to\infty$. For any $\lambda\in[1/2,1]$, we consider the problem
$$
\left\{
\begin{aligned}
    \left(a+b\int_{\R^N}|(-\triangle)^{\frac{\alpha}{2}}u|^2dx\right)(-\triangle)^{\alpha}u+V_{\infty} u =\lambda f(u) & \quad\mbox{in}\,\,\R^N, \\
    u\in H^\alpha(\R^N),\quad  u>0 & \quad\mbox{in}\,\, \R^N,
  \end{aligned}
\right.
$$
whose energy functional is defined by
\begin{equation*}
I_{\lambda}^\infty(u):=\frac{1}{2}\int_{\R^N}(a|(-\triangle)^{\frac{\alpha}{2}}u|^2+V_\infty u^2)dx+\frac{b}{4}\left(\int_{\R^N}|(-\triangle)^{\frac{\alpha}{2}}u|^2dx\right)^2-
\lambda\int_{\R^N}F(u)dx.
\end{equation*}

\noindent
We will use of the following Poho\v zaev type identity, whose
proof is similar as in \cite{Chang13}.

\begin{lemma}[Poho\v zaev identity]
	\label{Lem:pohozaev}
Let $u$ be a critical point of $I^\infty_\lambda$ in $\H$ for $\lambda\in
[\frac{1}{2},1]$. Then $P_\lambda(u)=0$,
\begin{equation}\label{eqn:4.below}
\aligned
P_\lambda(u)&:=
\frac{N-2\alpha}{2}\int_{\R^N}a|(-\triangle)^{\frac{\alpha}{2}}
u|^2dx+\frac{N-2\alpha}{2}b\left(\int_{\R^N}|(-\triangle)^{\frac{\alpha}{2}}
u|^2dx\right)^2\\
&\,\,\,\,\,\,+\frac{N}{2}\int_{\R^N}V_\infty u^2dx-N\lambda\int_{\R^N}F(u)dx.
\endaligned
\end{equation}
Notice that $P_\lambda(u)=\frac{d}{dt}I_\lambda^\infty(u(\cdot/t))\big|_{t=1}$.
\end{lemma}

\begin{lemma}\label{Lem:shlu1}
For $\lambda\in[\frac{1}{2},1]$, if $w_\lambda\in \H\setminus\{0\}$ solves $P_\lambda(w_\lambda)=0$,
then there exists $\gamma_\lambda\in C([0,1], \H)$ such that $\gamma_\lambda
(0)=0$, $I^{\infty}_{\lambda}(\gamma_\lambda(1))<0$, $w_\lambda\in \gamma_{\lambda}([0,1])$,
$0\not\in \gamma_\lambda((0,1])$ and
$$
\max_{t\in[0,1]}I_\lambda^\infty(\gamma_\lambda(t))=I_\lambda^\infty(w_\lambda).
$$
\end{lemma}
\begin{proof}
Note that
$$
\aligned
I_\lambda^\infty(w_\lambda(\cdot/t))=&\frac{t^{N-2\alpha}}{2}\int_{\R^N}a|(-\triangle)^{\frac{\alpha}{2}} w_\lambda|^2dx
+\frac{bt^{2N-4\alpha}}{4}\left(\int_{\R^N}|(-\triangle)^{\frac{\alpha}{2}} w_\lambda|^2dx\right)^2\\
&+\frac{t^N}{2}\int_{\R^N}V_\infty w_\lambda^2dx-t^N\lambda\int_{\R^N}F(w_\lambda)dx=0,
\endaligned
$$
which, by \eqref{eqn:4.below}, yields
$$
\lim_{t\to\infty}I_\lambda^\infty(w_\lambda(\cdot/t))<0.
$$
Then there is $t_0>0$ such that
$I_\lambda^\infty(w_\lambda(\cdot/t_0))<0$. Let $\gamma_\lambda(t)=w_\lambda(\cdot/tt_0))$
for $0<t\leq 1$ and $\gamma_\lambda(0)=0$.
Then $\gamma_\lambda\in C([0,1], \H)$, $w_\lambda\in \gamma_{\lambda}([0,1])$ and
$\max_{t\in[0,1]}I_\lambda^\infty(\gamma_\lambda(t))=I_\lambda^\infty(w_\lambda)$
as $t=t_0^{-1}$ is the unique maximum point of $t\mapsto I_\lambda^\infty(\gamma_\lambda(t))$ by Lemma~\ref{Lem:pohozaev}.
\end{proof}

\section{Behaviour of Palais-Smale sequences}
By Corollary \ref{bps}, for almost every $\lambda\in[1/2,1]$, there exists
a bounded Palais-Smale sequence $\{u_n\}_{n\in\N}\subset \H$ for $I_\lambda$
at the level $c_\lambda$. Then there exists
a subsequence of $\{u_n\}_{n\in\N}$, still denoted by $\{u_n\}_{n\in\N}$, such that $u_n\rightharpoonup u_0$ in $\H$
and $u_n\to u_0$ a.e.\ in $\R^N$ as $n\to\infty$. Let $v_n^1:=u_n-u_0$, then $v_n^1\rightharpoonup0$ in $\H$ and
$v_n^1\to 0$ a.e.\ in $\R^N$.

\subsection{Splitting lemmas} Let us set
$$
g(t):=f(t)-(t^{+})^{2_\alpha^*-1}, \quad\,\,\, G(t):=\int_{0}^t g(s)ds.
$$
In order to get the profile decomposition of $\{u_n\}_{n\in\N}$, we state the following
splitting lemmas.

	\begin{lemma}[Splitting lemma I]
		\label{Lem:global-0}
		We have
		\begin{equation}\label{eqn:4.3-1}
			\left|\int_{\R^N}(g(u_n)-g(u_0)-g(v_n^1))\varphi dx\right| \leq o_n(1)\|\varphi\|,
		\end{equation}
		where $o_n(1)\rg 0$ as $n\rg \iy$, uniformly for any $\varphi\in C_0^{\infty}(\R^N)$.
	\end{lemma}
\begin{proof}
		For each $n\geq 1$, there exists $\theta_n\in(0,1)$ such that
		\begin{equation}
		\label{eqn:4.3-2-}
			|g(u_n)-g(v_n^1)|\leq |g'(v_n^1+\theta_n u_0)||u_0|.
		\end{equation}
		In view of (f$_1$)-(f$_3$), for any $\epsilon>0$, there exists $\bar{D}>0$ such that
		\begin{equation}\label{eqn:4.3-2+}
			|g(t)|\leq \epsilon |t|^{2_\alpha^*-1},\quad \text{for}\,|t|\geq \bar{D}/2.
		\end{equation}
		Let $\Omega_n(\bar{D}):=\{x\in\R^N:\,|u_n(x)|\geq \bar{D}\}$ and for $r>0$, $B_r:=\{x\in\RN: |x|<r\}$, $B_r^c:=\RN\setminus B_r(0)$. Since $u_0\in \H$, we have $|B_R^c\cap\{|u_0(x)|\ge \bar{D}/2\}|\rg0$ as $R\rg\iy$. Then for $\e$ given as above, there exist $R>0$
			and $\om_R\subset \R^N$ with $|\om_R|\le \Lambda_\e$ such that $|u_0(x)|<{\bar{D}/2}$ for
			$x\in B_R^c\setminus\om_R,$  where $\Lambda_\e>0$ will be chosen later small enough.
			Then, by H\"{o}lder's inequality, \eqref{eqn:4.3-2-} and
			\eqref{eqn:4.3-2+}, we have
			\begin{equation}\label{eqn:4.3-3}
				\aligned
				&\int_{B_R^c\setminus\om_R}|g(u_n)-g(v_n^1)||\varphi|dx\\
				&\leq\int_{(B_R^c\setminus\om_R)\cap\Omega_n(\bar{D})}|g(u_n)-g(v_n^1)||\varphi|dx
				+\int_{(B_R^c\setminus\om_R)\cap\Omega_n^c(\bar{D})}|g(u_n)-g(v_n^1)||\varphi|dx\\
				&\leq\epsilon C(\|u_n\|_{2_\alpha^*}^{2_\alpha^*-1}+\|v^1_n\|_{2_\alpha^*}^{2_\alpha^*-1})
				\|\varphi\|
				+\max\limits_{|t|\leq 2\bar{D}}|g'(t)|\Big(\int_{B_R^c}u_0^2(x)dx\Big)^{1/2}\|\varphi\|.
				\endaligned
			\end{equation}
			It follows from (f$_1$) and (f$_2$) that, for $\varepsilon>0$ given, there exists $C_\varepsilon=C_\eps(f)>0$ such that
			\begin{equation}\label{eqn:4.3-3+}
				\aligned
				&\int_{\om_R}|g(u_n)-g(v_n^1)||\varphi|dx\\
				&\leq \varepsilon \int_{\om_R}(|u_n|^{2_\alpha^*-1}+|v_n^1|^{2_\alpha^*-1})|\varphi|dx
				+C_\varepsilon\int_{\om_R}(|u_n|+|v_n^1|)|\varphi|dx\\
				&\leq\epsilon C(\|u_n\|_{2_\alpha^*}^{2_\alpha^*-1}+\|v^1_n\|_{2_\alpha^*}^{2_\alpha^*-1})
				\|\varphi\|
				+C_\varepsilon 	|\om_R|^\frac{2\alpha}{N}(\|u_n\|_{2_\alpha^*}+\|v^1_n\|_{2_\alpha^*})\|\varphi\|_{2_\alpha^*}.
				\endaligned
			\end{equation}
			By \eqref{eqn:4.3-3} and \eqref{eqn:4.3-3+}, by choosing
			$\Lambda_\eps$ such that $C_\eps \Lambda_\eps^{2\alpha/N}\leq \eps$,
			there exists $C>0$ with
			\begin{equation}\label{eqn:4.3-4}
				\int_{B_R^c}|g(u_n)-g(v_n^1))||\varphi| dx\leq C\epsilon\|\varphi\|.
			\end{equation}
			Moreover,
			\begin{equation}\label{eqn:4.3-4-}
				\aligned
				\int_{B_R^c}|g(u_0)||\varphi|dx&\leq C\int_{B_R^c}|u_0||\varphi|dx+
				\int_{B_R^c}|u_0|^{2_\alpha^*-1}|\varphi|dx\\
				&\leq C\Big(\int_{B_R^c}|u_0|^2dx\Big)^{1/2}\|\varphi\|
				+C\Big(\int_{B_R^c}|u_0|^{2_{\alpha}^*}dx\Big)^{(2_\alpha^*-1)/2_\alpha^*}\|\varphi\|.
				\endaligned
			\end{equation}
			It follows from  \eqref{eqn:4.3-4} and \eqref{eqn:4.3-4-} that, for $\epsilon>0$ above, 
            we choose $R>0$ above large enough such that
			\begin{equation}\label{eqn:4.3-5}
				\left|\int_{B_R^c}(g(u_n)-g(u_0)-g(v_n^1))\varphi dx\right|\leq C\epsilon\|\varphi\|,
			\end{equation}
			where $C$ is independent of $n$, $\e$ and $\varphi\in C_0^{\infty}(\R^N)$.
		On the other hand,
		\begin{equation*}
			\int_{B_R}|g(u_n)-g(u_0)||\varphi| dx\leq\Big(\int_{B_R}|g(u_n)-g(u_0)|^{2_{\alpha}^*/(2_{\alpha}^*-1)}
			dx\Big)^{(2_{\alpha}^*-1)/2_{\alpha}^*}\Big(\int_{B_R}|\varphi|^{2_{\alpha}^*}\Big)^{1/2_{\alpha}^*}.
		\end{equation*}
		Observe that
		$$
		\lim_{t\to+\infty}\frac{g^{2_{\alpha}^*/(2_{\alpha}^*-1)}(t)}{t^{2_\alpha^*}}=
		\lim_{t\to 0^+}\frac{g^{2_{\alpha}^*/(2_{\alpha}^*-1)}(t)}{t^{2^*_\alpha/(2^*_\alpha-1)}}=0.
		$$
	Then
		$|g(u_n)-g(u_0)|^{2_{\alpha}^*/(2_{\alpha}^*-1)}\rg0$ in $L^1(B_R)$.
		Hence, we deduce
		\begin{equation}\label{eqn:4.3-7}
			\int_{B_R}|g(u_n)-g(u_0)||\varphi| dx \leq o_n(1)\|\varphi\|.
		\end{equation}
		Similarly, we also obtain that
		\begin{equation}\label{eqn:4.3-8}
			\int_{B_R}|g(v^1_n)|\varphi dx \leq o_n(1)\|\varphi\|,
		\end{equation}
		for any $\varphi\in C_0^{\infty}(\R^N)$. It follows from \eqref{eqn:4.3-5},
		\eqref{eqn:4.3-7} and \eqref{eqn:4.3-8} that \eqref{eqn:4.3-1} holds.
	\end{proof}
	
	\begin{lemma}[Splitting lemma II]\label{Lem:global-2}
		We have
		$$
		\left|\int_{\RN}\left(|u_n|^{2_\al^*-2}u_n-|u_0|^{2_\al^*-2}u_0-|v_n^1|^{2_\al^*-2}v_n^1\right)\vp dx\right|\leq o_n(1)\|\vp\|,
		$$
		where $o_n(1)\rg0$ as $n\rg\iy$, uniformly for any $\vp\in C_0^\iy(\RN)$.
	\end{lemma}
	\begin{proof}
		For any $\e>0$, there exists $R=R(\e)>0$ such that
		\begin{equation}\label{eqn:cl1}
			\aligned
			&\left|\int_{\RN\setminus B_R(0)}\left(|u_n|^{2_\al^*-2}u_n-|u_0|^{2_\al^*-2}u_0-|v_n^1|^{2_\al^*-2}v_n^1\right)\vp dx\right|\\
			&\le\left|\int_{\RN\setminus B_R(0)}\left(|u_n|^{2_\al^*-2}u_n-|v_n^1|^{2_\al^*-2}v_n^1\right)\vp dx\right|+\left|\int_{\RN\setminus B_R(0)}|u_0|^{2_\al^*-2}u_0\vp dx\right|\\
			&\le C\int_{\RN\setminus B_R(0)}\left(|u_n|^{2_\al^*-2}+|v_n^1|^{2_\al^*-2}\right)|u_0\vp|dx+\int_{\RN\setminus B_R(0)}|u_0|^{2_\al^*-1}|\vp|dx	\le C\e\|\vp\|.
			\endaligned
		\end{equation}
		On the other hand, for every $r>0$, we have
		\begin{equation*}
			\aligned
			&\left|\int_{B_R(0)}\left(|u_n|^{2_\al^*-2}u_n-|u_0|^{2_\al^*-2}u-|v_n^1|^{2_\al^*-2}v_n^1\right)\vp dx\right|\\
			&\le \int_{B_R(0)\cap\{|v_n^1|\le r\}}\left||u_n|^{2_\al^*-2}u_n-|u_0|^{2_\al^*-2}u_0-|v_n^1|^{2_\al^*-2}v_n^1\right|\vp dx \\
			&\,\,\,\,\,+\int_{B_R(0)\cap\{|v_n^1|\ge r\}}\left||u_n|^{2_\al^*-2}u_n-|u_0|^{2_\al^*-2}u_0-|v_n^1|^{2_\al^*-2}v_n^1\right|\vp dx
			=:I_1+I_2.
			\endaligned
		\end{equation*}
		Now, there exists $r=r(R)$ such that $r|B_R(0)|^{1/2_\al^*}\leq \eps$. Therefore, we have
		\begin{align}\label{eqn:cl3}
			I_1 &\le C\int_{B_R(0)\cap\{|v_n^1|\le r\}}\left(|u_n|^{2_\al^*-2}+|u_0|^{2_\al^*-2}+|v_n^1|^{2_\al^*-2}\right)|v_n^1\vp|dx \\
			& \le C r|B_R(0)|^{1/2_\al^*}\|\vp\|\le C \e\|\vp\|. \notag
		\end{align}
		For such $r,R$ fixed above, $u_n$ converges to $u$ in measure in $B_R(0)$, i.e.\
		$|B_R(0)\cap\{|v_n|\ge r\}|\to  0$ for $n\to\infty$. Therefore, for $n\geq 1$ large,
		\begin{align}\label{eqn:cl4}
			I_2 &\le C\int_{B_R(0)\cap\{|v^1_n|\ge r\}}\Big(|u_n|^{2_\al^*-2}
			+|v^1_n|^{2_\al^*-2}\Big)|u_0\vp|dx \\
			&\,\,\,\,\,\,\, +\int_{B_R(0)\cap\{|v^1_n|\ge r\}}|u_0|^{2_\al^*-1}|\vp|dx\le C\e\|\vp\|. \notag
		\end{align}
		Then \eqref{eqn:cl1}, \eqref{eqn:cl3} and \eqref{eqn:cl4} yield the assertion.
	\end{proof}
	
	\begin{lemma}[Splitting lemma III]
		\label{Lem:global-1}
		We have
		$$
		\int_{\R^N}f(u_n)u_ndx=\int_{\R^N}f(v_n^1)v_n^1dx+\int_{\R^N}f(u_0)u_0dx+o_n(1),
		$$
		where $o_n(1)\rg0$ as $n\rg\iy$. Furthermore
		\begin{equation*}
		\int_{\R^N}F(u_n)dx=\int_{\R^N}F(v_n^1)dx+\int_{\R^N}F(u_0)dx+o_n(1).
		\end{equation*}
	\end{lemma}
	\begin{proof}
		Since $f(t)=g(t)+t^{2_\al^*-1}$ for $t\geq 0$, by the standard Brezis--Lieb lemma, it suffices to prove
		$$
		\int_{\R^N}g(u_n)u_ndx=\int_{\R^N}g(v_n^1)v_n^1dx+\int_{\R^N}g(u_0)u_0dx+o_n(1),
		$$
		where $o_n(1)\rg0$ as $n\rg\iy$. Fixed
		$\epsilon>0$, there exists $C_\epsilon>0$ such that
		\begin{equation}\label{eqn:sub0}
			|g(t)|\leq \epsilon t^{2_\alpha^*-1}+ C_\epsilon t,\quad t\geq 0.
		\end{equation}
		Then there exists $R=R(\epsilon)>0$ large enough such that
		\begin{equation}\label{eqn:sub1}
			\aligned
			\left|\int_{\R^N}g(u_0)v^1_ndx\right|&\leq\int_{B_R}|g(u_0)v^1_n|dx+\int_{B_R^c}|g(u_0)v^1_n|dx\\
			&\leq\int_{B_R}(\epsilon|u_0|^{2_\alpha^*-1}+C_\epsilon|u_0|)|v^1_n|dx+
			\e(\|v_n^1\|_2+\|v^1_n\|_{2_\alpha^*})
			\leq C\epsilon+C_\eps o_n(1).
			\endaligned
		\end{equation}
		and
		\begin{equation}\label{eqn:sub2}
			\aligned
			\left|\int_{\R^N}g(v^1_n)u_0dx\right|&\leq\int_{B_R}|g(v^1_n)u_0|dx+\int_{B_R^c}|g(v^1_n)u_0|dx\\
			&\leq\int_{B_R}(\epsilon|v^1_n|^{2_\alpha^*-1}+C_\epsilon|v^1_n|)|u_0|dx+
			\int_{B_R^c}(\epsilon|v^1_n|^{2_\alpha^*-1}+C_\epsilon|v^1_n|)|u_0|dx\\
			&\leq C\epsilon+C_\eps o_n(1).
			\endaligned
		\end{equation}
		It follows from (\ref{eqn:sub1}), (\ref{eqn:sub2}) and Lemma \ref{Lem:global-0} that
		\begin{equation*}
			\aligned
			&\left|\int_{\R^N}(g(u_n)u_n-g(u_0)u_0-g(v^1_n)v^1_n)dx\right|\\
			&\leq \int_{\R^N}|(g(u_n)-g(u_0)-g(v^1_n))u_n| dx+\int_{\R^N}|g(v^1_n)u_0|dx+\int_{\R^N}|g(u_0)v^1_n|dx\\
			&\leq o_n(1)\|u_n\|+C\epsilon+C_\eps o_n(1).
			\endaligned
		\end{equation*}
		Letting $n\to\infty$ and $\eps\to 0^+$ completes the proof of the first assertion. The second assertion follows from the standard Brezis--Lieb lemma and
			\begin{equation*}
			\int_{\R^N}G(u_n)dx=\int_{\R^N}G(v_n^1)dx+\int_{\R^N}G(u_0)dx+o_n(1),
		\end{equation*}
		whose proof is left to the reader.
	\end{proof}

\subsection{Profile decomposition}
In the following, we give the profile decomposition
of $\{u_n\}_{n\in\N}$, which plays a crucial role in getting
the compactness. Since $c_\lambda>0$,  for some $\bar{B}>0$ we have
$$
\int_{\R^N}|(-\triangle)^{\frac{\alpha}{2}}u_n|^2dx\rightarrow \bar{B}^2,\qquad \text{as $n\rightarrow\infty$}.
$$
Now, for any $u\in \H$, let
$$
J_\lambda(u):=\frac{a+b\bar{B}^2}{2}\int_{\R^N}|(-\triangle)^{\frac{\alpha}{2}}
u|^2dx+\frac{1}{2}\int_{\R^N}V(x)|u|^2dx-\lambda\int_{\R^N}F(u)dx
$$
and
$$
J_\lambda^\infty(u):=\frac{a+b\bar{B}^2}{2}\int_{\R^N}|(-\triangle)^{\frac{\alpha}{2}}
u|^2dx+\frac{1}{2}\int_{\R^N}V_\infty|u|^2dx-\lambda\int_{\R^N}F(u)dx,
$$
which are respectively the corresponding functional of the following problems
$$
(a+b\bar{B}^2)(-\triangle)^\alpha {u}+V(x)u =f(u),\qquad
(a+b\bar{B}^2)(-\triangle)^\alpha {u}+V_\infty u =f(u),\,\,u\in \H.
$$
Here we point out that in contrast with the original problem (K),
the problems above are both {\it non Kirchhoff}.
Now we take advantage of this to get the profile decomposition of $\{u_n\}_{n\in\N}$.

\begin{lemma}[Profile decomposition]
	\label{Lem:global}
	Let $\{u_n\}_{n\in\N}\subset \H$ be the sequence mentioned above and
	assume that conditions {\rm (V$_1$)-(V$_3$)}, {\rm (f$_1$)-(f$_3$)} hold and $N<4\alpha$.
	Then $J'_\lambda(u_0)=0$,
	and there exist a number $k\in\N\cup\{0\}$, nontrivial critical points
	$w^1,\ldots,w^k$ of $J^\infty_{\lambda}$
	such that
	\begin{itemize}
		\item[\rm (i)] $|y_n^j|\rightarrow+\infty$,
		$|y_n^j-y_n^i|\rightarrow+\infty\quad$ if $i\neq j,\,1\leq i,j\leq k,\,n\rightarrow+\infty$,
		\item[\rm (ii)]   $c_\lambda+\frac{b\bar{B}^4}{4}=
		J_\lambda(u_0)+\sum\limits_{j=1}^{k}J^\infty_\lambda(w^j)$,
		\item[\rm (iii)]   $\|u_n-u_0-\sum\limits_{j=1}^{k}w^{j}(\cdot-y_n^j)\|\rightarrow0$,
		\item[\rm (iv)] $\bar{B}^2=\|(-\triangle)^{\frac{\alpha}{2}}
		u_0\|_2^2+\sum\limits_{j=1}^{k}\|(-\triangle)^{\frac{\alpha}{2}}w^j\|_2^2$.
	\end{itemize}
	Moreover, we agree that in the case $k=0$ the above holds without $w^j$.
In addition, if {\rm(V$_4$)} holds, then $k=0$ and $u_0\in H_{{\rm rad}}^\al(\RN)$.
\end{lemma}
\begin{proof}
	Observe that, from $I_\lambda(u_n)=c_\lambda+o_n(1)$ and
	$I'_\lambda(u_n)\to 0$ in $\H'$, we obtain
	$$
	J_\lambda(u_n)=c_\lambda+\frac{b\bar{B}^4}{4}+o_n(1), \,\,\,\quad
	J'_\lambda(u_n)\rightarrow 0\quad \text{in $\H'$}.
	$$
	Then, it is standard to get $J'_\lambda(u_0)\varphi=0$ for all $\varphi\in \H$. From Lemma \ref{Lem:global-1}, we get
	\begin{align*}
		\int_{\R^N}F(v_n^1)dx&=\int_{\R^N}F(u_n)dx-\int_{\R^N}F(u_0)dx+o_n(1), \\
\int_{\R^N}f(v_n^1)v_n^1dx&=\int_{\R^N}f(u_n)u_ndx-\int_{\R^N}f(u_0)u_0dx+o_n(1).
	\end{align*}
	It follows  that
	\begin{align}\label{eqn:4.0}
		J_\lambda(u_n)&=J_\lambda(v^1_n)+J_\lambda(u_0)+o_n(1), \\
	\label{eqn:4.1}
		J_\lambda '(v^1_n)v_n^1 &=J'_\lambda(u_n)u_n-J'_\lambda(u_0)u_0+o_n(1)=o_n(1).
	\end{align}
	On the other hand, by a slight variant of \cite[Proposition 4.1]{Chang13}, $u_0$ satisfies the Poh\v ozaev identity
	\begin{align*}
		\frac{N-2\alpha}{2}(a+b\bar{B}^2)\int_{\R^N}|(-\triangle)^{\frac{\alpha}{2}}
		u_0|^2dx&+\frac{1}{2}\int_{\R^N}\nabla V(x)\cdot x \, u_0^2dx\\
		&+\frac{N}{2}\int_{\R^N}V(x) u_0^2dx-N\lambda\int_{\R^N}F(u_0)dx=0.
	\end{align*}
	Then by (V$_1$) and $N<4\alpha$, we have
	$$
	\aligned
	NJ_\lambda(u_0)&=
	\alpha(a+b\bar{B}^2)\int_{\R^N}|(-\triangle)^\frac{\alpha}{2}u_0|^2dx
	-\frac{1}{2}\int_{\R^N}\nabla V(x)\cdot x \,u_0^2 dx\\
	&\geq\alpha(a+b\bar{B}^2)\int_{\R^N}|(-\triangle)^\frac{\alpha}{2}u_0|^2dx-\frac{1}{2}\|W\|_{\frac{N}{2\alpha}}\|u_0\|^2_{2_\alpha^*}\\
	&\geq\alpha(a+b\bar{B}^2)\int_{\R^N}|(-\triangle)^\frac{\alpha}{2}u_0|^2dx-a\alpha\int_{\R^N}|(-\triangle)^\frac{\alpha}{2}u_0|^2dx\\
	&=\alpha b\bar{B}^2\int_{\R^N}|(-\triangle)^\frac{\alpha}{2}u_0|^2dx>0,
	\endaligned
	$$
	which implies that
	\begin{equation}\label{eqn:4.2}
		J_\lambda(u_0)\geq\frac{b\bar{B}^2}{4}\int_{\R^N}|(-\triangle)^\frac{\alpha}{2}u_0|^2dx.
	\end{equation}
	We claim that one of the following conclusions holds for $v_n^1$:
	\begin{itemize}
		\item[\rm (v1)] $v_n^1\rightarrow0$ in $\H$, or
		\item[\rm (v2)]  there exist $r'>0$, $\sigma>0$ and a sequence $\{y^1_n\}_{n\in\N}\subset\R^N$ such that
		\begin{equation}\label{eqn:4.04}
			\liminf\limits_{n\rightarrow\infty}\int_{B_{r'}(y_n^1)}|v_n^1|^2dx\geq \sigma>0.
		\end{equation}
	\end{itemize}
	Indeed, suppose that \rm (v2) does not occur. Then for any $r>0$, we have
	\begin{equation*}
		\lim\limits_{n\rightarrow\infty}\sup\limits_{y\in\R^N}\int_{B_r(y)}|v_n^1|^2dx=0.
	\end{equation*}
	Therefore, it follows from Lemma \ref{Lem:P.lions1} that $v_n^1\rightarrow0$ in $L^s(\R^N)$ for $s\in(2,2_\alpha^*)$.
	It follows from (\ref{eqn:sub0}) that for any $\epsilon>0$, there exists $C_\epsilon>0$ such that
			$$
			\int_{\R^N}|g(v_n^1)v_n^1|dx\leq\epsilon \Big(\int_{\R^N}|v_n^1|^2+|v_n^1|^{2_\alpha^*}\Big)dx
			+C_\epsilon\int_{\R^N}|v_n^1|^{q}dx.
			$$
			So from $v_n^1\rightarrow0$ in $L^q(\R^N)$  and the arbitrariness of $\epsilon$,
			we can easily obtain that
			$$
			\int_{\R^N}f(v_n^1)v_n^1dx= \int_{\R^N}((v_n^1)^+)^{2^*_\alpha}dx+o_n(1).
			$$
	Furthermore, from $J'_\lambda(v_n^1)v_n^1=o_n(1)$ in \eqref{eqn:4.1}, we have
	\begin{equation}\label{eqn:4.3-}
		\|v_n^1\|^2+b\bar{B}^2\int_{\R^N}|(-\triangle)^{\frac{\alpha}{2}}v_n^1|^2dx
		=\lambda\|(v^1_n)^+\|_{2_\alpha^*}^{2_\alpha^*}+o_n(1).
	\end{equation}
	In view of conditions (V$_2$)-(V$_3$), we can
	check that $V_\infty>0$. And so we can also get
			$$
			\int_{\R^N}V(x) |v_n^1|^2dx=\int_{\R^N}V^+(x) |v_n^1|^2dx+o_n(1),
			$$
			which, together with the definition of $S_\alpha$ and \eqref{eqn:4.3-}, implies that
			\begin{equation}\label{eqn:3.010}
				 aS_\alpha\left(\int_{\R^N}|v_n^1|^{2_\alpha^*}dx\right)^{\frac{2}{2_\alpha^*}}+bS_\alpha^2\left(\int_{\R^N}|v_n^1|^{2_\alpha^*}dx\right)^{\frac{4}{2_\alpha^*}}
				\leq \lambda\int_{\R^N}|v_n^1|^{2_\alpha^*}dx+o_n(1).
			\end{equation}
	Let $\ell\geq0$ be such that $\int_{\R^N}|v_n^1|^{2_\alpha^*}dx\rightarrow \ell^{N}$.
	If $\ell>0$, then
	it follows from (\ref{eqn:3.010}) that
	$$
	K'(\ell)=\frac{(N-2\alpha)\ell^{-1}}{2}(aS_\alpha \ell^{N-2\alpha}+bS_\alpha^2 \ell^{2N-4\alpha}-\lambda \ell^{N})\leq0,
	$$
	where $K$ has been defined in Lemma \ref{Lem:level}. This also implies that
	$\ell\geq T$ ($T$ is the unique maximum point of $K$).  On the other hand, by \eqref{eqn:4.0} and \eqref{eqn:4.2}, we have
	$$
	\aligned
	c_\lambda+\frac{b\bar{B}^4}{4}&=\int_{\R^N}\left(\frac{a+b\bar{B}^2}{2}|(-\triangle)^{\frac{\alpha}{2}}v_n^1|^2+\frac{1}{2}V(x) |v_n^1|^2-\frac{\lambda}{2_\alpha^*}((v_n^1)^+)^{2_\alpha^*}\right)dx+J_\lambda(u_0)+o_n(1)\\
	&\geq\int_{\R^N}\left(\Big(\frac{a}{2}+\frac{b\bar{B}^2}{4}\Big)|(-\triangle)^{\frac{\alpha}{2}}v_n^1|^2
	+\frac{1}{2}V(x)|v_n^1|^2
	-\frac{\lambda}{2_\alpha^*}((v_n^1)^+)^{2_\alpha^*}\right)dx+\frac{b\bar{B}^4}{4}+o_n(1),
	\endaligned
	$$
	which, together with (\ref{eqn:4.3-}) and the definition of $S_\alpha$,
	implies that
	$$
	\aligned
	c_\lambda&\ge\left(\frac{1}{2}-\frac{1}{2_\alpha^*}\right)a\int_{\R^N}|(-\triangle)^{\frac{\alpha}{2}}v_n^1|^2dx
	+\left(\frac{1}{4}-\frac{1}{2_\alpha^*}\right)b\left(\int_{\R^N}|(-\triangle)^{\frac{\alpha}{2}}v_n^1|^2dx\right)^2
	+o_n(1)\\
	&\geq\left(\frac{1}{2}-\frac{1}{2_\alpha^*}\right)aS_\alpha\left(\int_{\R^N}|v_n^1|^{2_\alpha^*}dx\right)^{\frac{2}{2_\alpha^*}}+
	\left(\frac{1}{4}-\frac{1}{2_\alpha^*}\right)bS^2_\alpha\left(\int_{\R^N}|v_n^1|^{2_\alpha^*}dx\right)^{\frac{4}{2_\alpha^*}} +o_n(1).
	\endaligned
	$$
	Thus, combining $\int_{\R^N}|v_n^1|^{2_\alpha^*}dx\rightarrow \ell^{N}$ and $\ell\geq T$, $K'(T)= 0$, we have
	$$
	\aligned
	c_\lambda&\ge\left(\frac{1}{2}-\frac{1}{2_\alpha^*}\right)aS_\alpha \ell^{N-2\alpha}
	+\left(\frac{1}{4}-\frac{1}{2_\alpha^*}\right)bS^2_\alpha \ell^{2N-4\alpha}\\
	&\geq\left(\frac{1}{2}-\frac{1}{2_\alpha^*}\right)aS_\alpha T^{N-2\alpha}
	+\left(\frac{1}{4}-\frac{1}{2_\alpha^*}\right)bS^2_\alpha T^{2N-4\alpha}\\
	&=\frac{1}{2}aS_\alpha T^{N-2\alpha}+\frac{1}{4}bS^2_\alpha T^{2N-4\alpha}-\frac{\lambda}{2_\alpha^*}T^{N}
	=c_\lambda^*,
	\endaligned
	$$
	contradicting $c_\lambda<c_\lambda^*$. Hence, $\ell=0$. It follows from (\ref{eqn:4.3-}) that $\|v_n^1\|\rightarrow0$, that is, $u_n\rightarrow u_0$ in $\H$. Then Lemma \ref{Lem:global} hold with $k=0$ if \rm (v2) does not occur.
	 In particular, if we assume {\rm(V$_4$)} holds, then by Corollary \ref{bps}, $\|u_n-|u_n|^*\|_{2^*_\alpha}\to 0$. Obviously, $\{|u_n|^*\}_{n\in\N}\subset H_{{\rm rad}}^\al(\RN)$ is bounded and $\|u_n-|u_n|^*\|_q\to 0$ for $q\in(2,2_\al^*)$. Since $\{|u_n|^*\}_{n\in\N}$ has a strongly convergent subsequence in $L^q(\RN)$ for $q\in(2,2_\al^*)$, without loss of generality, we assume that $u_n\rg u_0$ in $L^q(\RN)$ for $q\in(2,2_\al^*)$ and $u_0=u_0^*$. As a consequence, \rm (v2) does not hold and as above, $u_n\rightarrow u_0$ in $\H$.
	
	 In the following,
	otherwise, suppose that \rm (v2) holds, that is \eqref{eqn:4.04} holds. Consider
	$v_n^1(\cdot+y_n^1)$. The boundedness of $\{v_n^1\}_{n\in\N}$ and \eqref{eqn:4.04} imply that $v_n^1(\cdot+y_n^1)\rightharpoonup w^1\neq0$
	in $\H$.  Thus, it follows from $v_n^1\rightharpoonup 0$ in $\H$ that $\{y_n^1\}_{n\in\N}$ is unbounded and, up to a subsequence, $|y^1_n|\to+\infty$.
	Let us prove that $(J_\lambda^{\infty})'(w^1)=0$. It suffices to show that $(J_\lambda^{\infty})'(v_n^1(\cdot+y_n^1))\varphi\rightarrow0$ for any $\varphi\in C_0^{\infty}(\R^N)$.
	Combining Lemma~\ref{Lem:global-0}  and Lemma \ref{Lem:global-2}, we obtain
	\begin{equation*}
		|J'_{\lambda}(u_n)\varphi-J'_{\lambda}(u_0)\varphi-J'_{\lambda}(v_n^1)\varphi|\leq o_n(1)\|\varphi\|,\quad\,\,\, \forall \varphi\in C_0^{\infty}(\R^N),
	\end{equation*}
	which implies that $|J'_{\lambda}(v_n^1)\varphi|\leq o_n(1)\|\varphi\|$, for all $\varphi\in C_0^{\infty}(\R^N)$, as $n\to\infty$. Notice that
	\begin{equation*}
		\aligned
		 J_{\lambda}'(v^1_n)\varphi(\cdot-y_n^1)=&\frac{C(n,\alpha)}{2}(a+b\bar{B}^2)\int_{\R^{2N}}\frac{(v_n^1(x)-v_n^1(y))(\varphi(x-y_n^1)-\varphi(y-y_n^1))}{|x-y|^{N+2\alpha}}dxdy\\
		&+\int_{\R^N}V(x)v^1_n(x)\varphi(x-y_n^1)dx-\lambda\int_{\R^N}g(v_n^1(x))\varphi(x-y_n^1)dx\\
		&-\lambda\int_{\R^N}((v_n^1(x))^+)^{2_\alpha^*-1}\varphi(x-y_n^1)dx
		=o_n(1)\|\varphi(\cdot-y_n^1)\|=o_n(1)\|\varphi\|.
		\endaligned
	\end{equation*}
	Thus, as $n\to\infty$, it follows that
	\begin{equation}\label{eqn:4.5}
		\aligned
		&\frac{C(n,\alpha)}{2}(a+b\bar{B}^2)\int_{\R^{2N}}\frac{(v_n^1(x+y_n^1)-v_n(y+y_n^1))(\varphi(x)-\varphi(y))}{|x-y|^{N+2\alpha}}dxdy+
		\\&\int_{\R^N}V(x+y_n^1)v^1_n(x+y_n^1)\varphi(x)dx
		-\lambda\int_{\R^N}g(v_n^1(x+y_n^1))\varphi(x)dx\\
		&-\int_{\R^N}((v_n^1(x+y_n^1))^+)^{2_\alpha^*-1}\varphi(x)dx=o_n(1)\|\varphi\|.
		\endaligned
	\end{equation}
	Since $|y_n^1|\rightarrow\infty$ and $\varphi\in C_0^\infty(\R^N)$,
	we obtain
	\begin{equation}\label{eqn:4.6}
		\int_{\R^N}(V(x+y_n^1)-V_\infty)v^1_n(x+y_n^1)\varphi(x)dx\rightarrow0.
	\end{equation}
	Thus, combining (\ref{eqn:4.5}) and (\ref{eqn:4.6}), we have for any $\varphi\in C_0^\infty(\R^N)$,
	$$
	\aligned
	 (J^\infty_{\lambda})'(v^1_n(\cdot+y_n^1))\varphi=&\frac{C(n,\alpha)}{2}(a+b\bar{B}^2)\int_{\R^{2N}}\frac{(v_n^1(x+y_n^1)-v_n(y+y_n^1))(\varphi(x)-\varphi(y))}{|x-y|^{N+2\alpha}}dxdy
	\\
	&+\int_{\R^N}V_\infty v^1_n(x+y_n^1)\varphi(x)dx
	-\lambda\int_{\R^N}g(v_n^1(x+y_n^1))\varphi(x)dx\\
	&-\lambda\int_{\R^N}((v_n^1(x+y_n^1))^+)^{2_\alpha^*-1}\varphi(x)dx=o_n(1).
	\endaligned
	$$
	Then, $(J_\lambda^{\infty})'(w^1)=0$. Finally, let us set
	\begin{equation}\label{eqn:4.11'}
		v_n^2(x)=v_n^1(x)-w^1(x-y_n^1),
	\end{equation}
	then $v_n^2\rightharpoonup 0$ in $\H$. 
	Since $V(x)\to V_\infty$ as $|x|\to\infty$ and $v_n^1\rightarrow0$ strongly in $L^2_{{\rm loc}}(\R^N)$, we have
	\begin{equation*}
	\int_{\R^N}(V(x)-V_\infty)(v_n^1)^2dx=o_n(1).
	\end{equation*}
	It follows that
	\begin{align}\label{eqn:4.12}
		\int_{\R^N}V(x)|v_n^2|^2dx& =\int_{\R^N}V(x)|v_n^1|^2dx+\int_{\R^N}V(x+y_n^1)|w^1(x)|^2dx \\
		&-2\int_{\R^N}V(x+y_n^1)v_n^1(x+y_n^1)w^1(x)dx \notag \\
		&=\int_{\R^N}V_\infty|u_n|^2dx-\int_{\R^N}V_\infty|u_0|^2dx -\int_{\R^N}V_\infty |w^1|^2dx+o_n(1)  \notag \\ 
		&=\int_{\R^N}V(x)|u_n|^2dx-\int_{\R^N}V(x)|u_0|^2dx -\int_{\R^N}V_\infty |w^1|^2dx+o_n(1), \notag
	\end{align}
	\begin{equation}\label{eqn:4.13}
		\left\{
		\begin{array}{ll}
			\| v_n^2\|^2=\| u_n\|^2-\|u_0\|^2-\| w^1(\cdot-y_n^1)\|^2+o_n(1), \\
			\|v_n^2\|_{2_\alpha^*}^{2_\alpha^*}=\|u_n\|_{2_\alpha^*}^{2_\alpha^*}-\|u_0\|_{2_\alpha^*}^{2_\alpha^*}-\|w^1\|_{2_\alpha^*}^{2_\alpha^*}+o_n(1),
		\end{array}
		\right.
	\end{equation}
	\begin{equation}\label{eqn:4.14}
		\aligned
		&\int_{\R^N}G(v_n^2)dx=\int_{\R^N}G(u_n)dx-\int_{\R^N}G(u_0)dx-\int_{\R^N}G(w^1)dx.
		\endaligned
	\end{equation}
	Similar to (\ref{eqn:4.3-1}), we also have
	\begin{equation}\label{eqn:4.14-0}
		\aligned
		&\int_{\R^N}g(v_n^2)\varphi dx=\int_{\R^N}g(u_n)\varphi dx-\int_{\R^N}g(u_0)\varphi dx-\int_{\R^N}g(w^1(\cdot-y_n^1))\varphi dx+o_n(1)\|\varphi\|,
		\endaligned
	\end{equation}
	for any $\varphi\in C_0^\infty(\R^N)$.
	Combining (\ref{eqn:4.12}), (\ref{eqn:4.13}), (\ref{eqn:4.14}) and (\ref{eqn:4.14-0}), we deduce that
	\begin{equation*}
		\aligned
		(1)\quad&J_\lambda(v_n^2)=J_\lambda(u_n)-J_\lambda(u_0)-J^{\infty}_\lambda(w^1)+o_n(1),\\
		(2)\quad&J'_\lambda(v_n^2)\varphi=
		J'_\lambda(u_n)\varphi-J'_\lambda(u_0)\varphi-(J^{\infty}_\lambda)'(w^1(\cdot -y_n^1))\varphi+o_n(1)\|\varphi\|=o_n(1)\|\varphi\|,\\
		(3)\quad&J_\lambda^{\infty}(v_n^2)=J_\lambda^{\infty}(v_n^1)-J_\lambda^{\infty}(w^1)+o_n(1)
		\endaligned
	\end{equation*}
	for any $\varphi\in C^\infty_0(\R^N)$. Thus, we get
	$$
	J_\lambda(v_n^2)=c_\lambda+\frac{b\bar{B}^4}{4}-J_\lambda(u_0)-J^{\infty}_\lambda(w^1)+o_n(1)
	< c_\lambda^*+\frac{b\bar{B}^4}{4}.
	$$
	Remark that one of (v1) and (v2) holds for $v_n^2$. If $v_n^2\rightarrow0$ in
	$\H$, then Lemma \ref{Lem:global} holds with $k=1$. Otherwise, $\{v_n^2\}$ is non-vanishing,
	that is, (v2) holds for $v_n^2$. Similarly, we repeat the arguments. By iterating
	this procedure we obtain sequences of points $\{y_n^j\}\subset\R^N$ such that $|y_n^j|\rightarrow+\infty$,
	$|y_n^j-y_n^i|\rightarrow+\infty$ if $i\neq j$ as $n\rightarrow+\infty$ and $v_n^j=v_n^{j-1}-w^{j-1}(x-y_n^{j-1})$ (like (\ref{eqn:4.11'})) with $j\geq2$
	such that
	$
	v_n^j\rightharpoonup 0\,\,\mbox{in}\,\, \H,\,\,(J^\infty_\lambda)'(w^j)=0
	$.
	Using the properties of the weak convergence, we have
	\begin{equation}\label{eqn:4.15}
		\aligned
		(a)\quad &\|u_n\|^2-\|u_0\|^2-\sum_{j=1}^{k}\|w^j(\cdot-y_n^j)\|^2=\|u_n-u_0-\sum_{j=1}^{k}w^{j}(\cdot-y_n^{j})\|^2+o(1),\\
		(b)\quad& J_\lambda(u_n)\rightarrow J_\lambda(u_0)+\sum_{j=1}^{k}J_\lambda^\infty(w^{j})+J^\infty_\lambda(v_n^{k+1}).
		\endaligned
	\end{equation}
	Note that there is $\rho>0$ such that $\|w\|\geq \rho$ for every nontrivial
	critical point $w$ of $J_\lambda^\infty$ and $\{u_n\}_{n\in\N}$ is bounded in $\H$. By (\ref{eqn:4.15})(a),
	the iteration stops at some $k$. That is, $v_n^{k+1}\rightarrow0$ in $\H$.
	The proof is complete.
\end{proof}


\section{Proof of the main results}
\label{sec3}
In order to obtain the existence of ground state solutions of problem (K),
our strategy is that we firstly obtain the existence nontrivial solutions of
the perturbed problem, then as $\lambda$ goes to $1$, we get a nontrivial solution of the original problem.
Finally, thanks to the profile decomposition of the
(PS)-sequence, we obtain the existence
of ground state solutions of problem (K).

\subsection{Nontrivial critical points of $I_\lambda$}

\begin{lemma}\label{Lem:compact}
	Assume that {\rm(V$_1$)-(V$_3$)} and {\rm(f$_1$)-(f$_3$)} hold. For almost every $\lambda\in[1/2,1]$,
	there exists $u_\lambda\in \H\setminus\{0\}$ such that $I_\lambda(u_\lambda)= c_\lambda$ and $I_\lambda'(u_\lambda)=0$. In addition, if {\rm(V$_4$)} holds, then $u_\la\in H_{{\rm rad}}^\al(\RN)$.
\end{lemma}
\begin{proof}
For almost all $\lambda\in[1/2,1]$, there is a bounded sequence
$\{u_n\}_{n\in\N}\subset \H$ such that $I_\lambda(u_n)\rightarrow c_\lambda$, $I_\lambda'(u_n)\rightarrow0$.
From Lemma \ref{Lem:global}, up to a subsequence, there exist $u_0\in \H$
and $\bar{B}>0$ such that
$$
u_n\rightharpoonup u_0\quad\mbox{in}\,\,\H,\quad\int_{\R^N}|(-\triangle)^{\frac{\alpha}{2}} u_n|^2dx\rightarrow \bar{B}^2,\,\mbox{as}\,n\rightarrow\infty
$$
and $J'_\lambda(u_0)=0$.
Furthermore, there exist $k\in\N\cup\{0\}$,
nontrivial critical points $w^1,\ldots,w^{k}$ of $J_\lambda^\infty$
and $k$ sequences of points
$\{y_n^j\}\subset\R^N$, $1\leq j\leq k$, such that
\begin{equation}\label{eqn:4.16-}
	\aligned
	\left\|u_n-u_0-\sum\limits_{j=1}^{k}w^{j}(\cdot-y_n^j)\right\|\rightarrow0,\quad c_\lambda+\frac{b\bar{B}^4}{4}=
	J_\lambda(u_0)+\sum\limits_{j=1}^{k}J^\infty_\lambda(w^j)
	\endaligned
\end{equation}
and
\begin{equation}\label{eqn:4.16}
	\bar{B}^2=\|(-\triangle)^{\frac{\alpha}{2}}
	u_0\|_2^2+\sum\limits_{j=1}^{k}\|(-\triangle)^{\frac{\alpha}{2}}w^j\|_2^2.
\end{equation}
Now we claim that if $u_0\not=0$, then by $N<4\alpha$,
\begin{equation}\label{eqn:4.17}
	J_\lambda(u_0)>\frac{b\bar{B}^2}{4}\int_{\R^N}|(-\triangle)^{\frac{\alpha}{2}} u_0|^2dx.
\end{equation}
Indeed, since $J'_\lambda(u_0)=0$, similar as in \cite{Chang13}, we get
$$
\aligned
\bar{P}_\lambda(u_0):=&\frac{N-2\alpha}{2}(a+b\bar{B}^2)\int_{\R^N}|(-\triangle)^{\frac{\alpha}{2}}
u_0|^2dx
+\frac{N}{2}\int_{\R^N}V(x)u_0^2dx\\
&+\frac{1}{2}\int_{\R^N}(\nabla
V(x),x)u_0^2dx-N\lambda\int_{\R^N}F(u_0)dx=0.
\endaligned
$$
By hypothesis (V$_1$) we have
$$
J_\lambda(u_0)=\frac{\alpha}{N}(a+b\bar{B}^2)\int_{\R^N}|(-\triangle)^{\frac{\alpha}{2}}u_0|^2dx-\frac{1}{2N}\int_{\R^N}\nabla
V(x)\cdot x \,u_0^2dx
>\frac{\alpha}{N}b\bar{B}^2\int_{\R^N}|(-\triangle)^{\frac{\alpha}{2}}u_0|^2dx,
$$
which implies that (\ref{eqn:4.17}) holds.
For each nontrivial critical point $w^j,\,(j=1,...,k)$ of $J^\infty_\lambda$,
$$
\frac{N-2\alpha}{2}(a+b\bar{B}^2)\int_{\R^N}|(-\triangle)^{\frac{\alpha}{2}}
w^j|^2dx+\frac{N}{2}\int_{\R^N}V_\infty|w^j|^2dx-N\lambda\int_{\R^N}F(w^j)dx=0.
$$
Then it follows from (\ref{eqn:4.16}) that
$$
\aligned
&\frac{a(N-2\alpha)}{2}\int_{\R^N}|(-\triangle)^{\frac{\alpha}{2}} w^j|^2dx+
\frac{b(N-2\alpha)}{2}\left(\int_{\R^N}|(-\triangle)^{\frac{\alpha}{2}} w^j|^2dx\right)^2\\
&+\frac{N}{2}\int_{\R^N}V_\infty|w^j|^2dx-N\lambda\int_{\R^N}F(w^j)dx\leq0.
\endaligned
$$
Then there exists $t_j\in(0,1]$ such that
\begin{equation}\label{eqn:4.17-}
	\aligned
	&\frac{at^{N-2\alpha}_j}{2}(N-2\alpha)\int_{\R^N}|(-\triangle)^{\frac{\alpha}{2}} w^j|^2dx
	+\frac{bt^{2N-4\alpha}_j}{2}(N-2\alpha)\left(\int_{\R^N}|(-\triangle)^{\frac{\alpha}{2}} w^j|^2dx\right)^2\\
	&+\frac{Nt_j^N}{2}\int_{\R^N}V_\infty|w^j|^2dx-Nt_j^N\lambda\int_{\R^N}F(w^j)dx=0.
	\endaligned
\end{equation}
That is, $w^j(\cdot/t_j)$ satisfies the identity $P_\lambda(u)=0$ and it follows from Lemma \ref{Lem:shlu1} that there exists $\gamma_\lambda\in C([0,1],\H)$ such that $\gamma_\lambda
(0)=0$, $I^{\infty}_{\lambda}(\gamma_\lambda(1))<0$, $w^j\in \gamma_{\lambda}([0,1])$ and
$$
I_\lambda^\infty(w^j(\cdot/t_j))=\max_{t\in[0,1]}I_\lambda^\infty(\gamma_\lambda(t)).
$$
By hypothesis (V$_2$), we have
$\max_{t\in[0,1]}I_\lambda^\infty(\gamma_\lambda(t))\ge\max_{t\in[0,1]}I_\lambda(\gamma_\lambda(t)),$
which, by the definition of $c_\lambda$, implies that $I_\lambda^\infty(w^j(\frac{\cdot}{t_j}))\ge c_\lambda$.
In particular, if $V(x)\not\equiv V_\infty$, then
\begin{equation}\label{eqn:4.17--}
	I_\lambda^\infty(w^j(\cdot/t_j))>c_\lambda.
\end{equation}
So by (\ref{eqn:4.17-}) we have
\begin{equation}\label{eqn:4.17+}
	\aligned
	J^\infty_\lambda(w^j)&=J^\infty_\lambda(w^j)-\frac{1}{N}P_\lambda(w^j)
	=(a+b\bar{B}^2)\left(\frac{1}{2}-\frac{1}{2_\alpha^*}\right)\int_{\R^N}|(-\triangle)^{\frac{\alpha}{2}} w^j|^2dx\\
	&\geq\left(\frac{1}{2}-\frac{1}{2_\alpha^*}\right)a\int_{\R^N}|(-\triangle)^{\frac{\alpha}{2}} w^j(\frac{x}{t_j})|^2dx\\
	&\,\,\,\,+\left(\frac{1}{4}-\frac{1}{2_\alpha^*}\right)b\left(\int_{\R^N}|(-\triangle)^{\frac{\alpha}{2}} w^j(\frac{x}{t_j})|^2dx\right)^2
	+\frac{b\bar{B}^2}{4}\int_{\R^N}|(-\triangle)^{\frac{\alpha}{2}} w^j|^2dx\\
	&=I^\infty_\lambda( w^j(\frac{\cdot}{t_j}))-\frac{1}{N}P_\lambda( w^j(\frac{\cdot}{t_j}))+\frac{b\bar{B}^2}{4}\int_{\R^N}|(-\triangle)^{\frac{\alpha}{2}} w^j|^2dx\\
	&=I^\infty_\lambda( w^j(\frac{\cdot}{t_j}))+\frac{b\bar{B}^2}{4}\int_{\R^N}|(-\triangle)^{\frac{\alpha}{2}} w^j|^2dx
	\endaligned
\end{equation}
and then we conclude that
$$
J^\infty_\lambda(w^j)\ge c_\lambda+\frac{b\bar{B}^2}{4}\int_{\R^N}|(-\triangle)^{\frac{\alpha}{2}} w^j|^2dx,
$$
where the inequality is strict if $V(x)\not\equiv V_\infty$. Then
by formulas \eqref{eqn:4.16}-\eqref{eqn:4.17},
\begin{equation}\label{eqn:4.18}
	\aligned
	c_\lambda+\frac{b\bar{B}^4}{4}&=J_\lambda(u_0)+\sum\limits_{j=1}^{k}J^\infty_\lambda(w^j)\geq kc_\lambda+\frac{b\bar{B}^4}{4},
	\endaligned
\end{equation}
where the inequality is strict if $V(x)\not\equiv V_\infty$. It follows that either $k=0$ or $k=1$. If $k=0$, we are done.
In particular, if {\rm(V$_4$)} holds, then $k=0$ and $u_0\in H_{{\rm rad}}^\al(\RN)$. Then $I_\la(u_0)=J_\lambda(u_0)-\frac{b\bar{B}^4}{4}=c_\la$ and $I_\la'(u_0)=J_\la'(u_0)=0$. We are done. If $k=1$ and $u_0\not=0$,
		then it follows from (\ref{eqn:4.17}) and (\ref{eqn:4.18})
		$$
		c_\lambda+\frac{b\bar{B}^4}{4}=J_\lambda(u_0)+\sum\limits_{j=1}^{k}J^\infty_\lambda(w^j)> c_\lambda+\frac{b\bar{B}^4}{4},
		$$
		which is a contradiction.
		So $u_0=0$, $k=1$ and $\bar{B}^2=\|(-\triangle)^{\frac{\alpha}{2}}w^1\|_2^2$. It follows from (\ref{eqn:4.16-}) and (\ref{eqn:4.16}) that
		$$
		\aligned
		\|u_n-w^1(\cdot-y_n^1)\|\rightarrow0,\,\,\,c_\lambda+\frac{b\bar B^4}{4}=J^\infty_\lambda(w^1).
		\endaligned
		$$
		Since if $V(x)\not\equiv V_\infty$, then by (\ref{eqn:4.17--})-(\ref{eqn:4.17+}),
		$$
		J^\infty_\lambda(w^1)\ge I^\infty_\lambda( w^1(\frac{\cdot}{t_1}))+\frac{b\bar{B}^4}{4}>c_\la+\frac{b\bar{B}^4}{4},
		$$
		which is a contradiction. Then $V(x)\equiv V_\infty$ and $u_n\rightarrow w^1$ strongly in $\H$. Therefore,
		$w^1$ is a nontrivial critical point of $I^\infty_\lambda$ and
		$I^\infty(u_0)=c_\lambda$.
		The proof is completed.
	\end{proof}

\subsection{Completion of the proof}
Choosing a sequence $\{\lambda_n\}_{n\in\N}\subset[\frac{1}{2},1]$ satisfying $\lambda_n\rightarrow1$,
we find a sequence of nontrivial critical points $\{u_{\lambda_n}\}_{n\in\N}$ (still denoted by $\{u_n\}_{n\in\N}$) of $I_{\lambda_n}$ and $I_{\lambda_n}(u_n)= c_{\lambda_n}$.
In particular, if {\rm(V$_4$)} holds, then $\{u_n\}_{n\in\N}\subset H_{{\rm rad}}^\al(\RN)$.
Now we show that $\{u_n\}$ is bounded in $\H$.
Remark that
$u_n$ satisfies the Poho\v zaev identity as follows
$$
\aligned
&\frac{N-2\alpha}{2}\int_{\R^N}a|(-\triangle)^{\frac{\alpha}{2}}
u_n|^2dx+\frac{N-2\alpha}{2}b\left(\int_{\R^N}|(-\triangle)^{\frac{\alpha}{2}}
u_n|^2dx\right)^2\\
&+\frac{N}{2}\int_{\R^N}V(x)u_n^2dx+\frac{1}{2}\int_{\R^N}\nabla
V(x)\cdot x \,u_n^2dx-N\lambda\int_{\R^N}F(u_n)dx=0.
\endaligned
$$
It follows that
$$
NI_{\lambda_n}(u_n)=\alpha\int_{\R^N}a|(-\triangle)^{\frac{\alpha}{2}}u_n|^2dx+\left(\alpha-\frac{N}{4}\right)b\left(\int_{\R^N}|(-\triangle)^{\frac{\alpha}{2}}
u_n|^2dx\right)^2-\frac{1}{2}\int_{\R^N}\nabla
V(x)\cdot x\, u_n^2dx.
$$
Since $c_\lambda^*$ is continuous on $\lambda$, $I_{\lambda_n}(u_n)= c_{\lambda_n}+o_n(1)<c_{\lambda_n}^*$.
It follows from (V$_1$) that there is a positive number $\kappa\in(0,2a\alpha)$ such that
$\|W\|_{\frac{N}{2\alpha}}\leq \kappa S_{\alpha}$. Hence,
$$
\left(a\alpha-\frac{\kappa}{2}\right)\int_{\R^N}|(-\triangle)^{\frac{\alpha}{2}}u_n|^2dx\leq NI_{\lambda_n}(u_n),
$$
which implies that $\int_{\R^N}a|(-\triangle)^{\frac{\alpha}{2}}u_n|^2dx$ is bounded from above. By (V$_3$), (f$_1$)-(f$_2$) and $I'_{\lambda_n}(u_n)u_n=0$, there is $\nu>0$ such that for any $\epsilon>0$,
there exists $C_\epsilon>0$ with
$$
\nu\int_{\R^N}u_n^2 dx\leq \int_{\R^N}a|(-\triangle)^{\frac{\alpha}{2}}u_n|^2dx+\int_{\R^N}V(x)u_n^2dx\leq
\epsilon\int_{\R^N}u_n^2dx+C_\epsilon\int_{\R^N}u_n^{2^*_\alpha}dx,
$$
which yields that $\{u_n\}_{n\in\N}$ is bounded in $L^2(\R^N)$.
Then $\{u_n\}_{n\in\N}$ is bounded in $\H$. By Theorem \ref{Thm:MP},
$$
\lim\limits_{n\rightarrow\infty}I(u_n)=\lim\limits_{n\rightarrow\infty}\left(I_{\lambda_n}(u_n)
+(\lambda_n-1)\int_{\R^N}F(u_n)dx\right)=\lim\limits_{n\rightarrow\infty}c_{\lambda_n}=c_1
$$
and for any $\varphi\in C_0^{\infty}(\R^N)$,
$$
\lim\limits_{n\rightarrow\infty}I'(u_n)\varphi=\lim\limits_{n\rightarrow\infty}\left(I'_{\lambda_n}(u_n)\varphi+(\lambda_n-1)\int_{\R^N}f(u_n)\varphi dx\right)
=0.
$$
That is, $\{u_n\}_{n\in\N}$ is a bounded Palais-Smale sequence for $I$ at level $c_1$.
Then by Lemma \ref{Lem:compact}, there is
a nontrivial critical point $u_0\in \H$ (radial, if {\rm(V$_4$)} holds)
for $I$ and $I(u_0)=c_1.$ Set 
$$
\nu=\inf\{I(u):u\in \H\setminus\{0\},\,I'(u)=0\}.
$$
Of course $0<\nu\leq I(u_0)= c_1<\infty$.\ By the definition of $\nu$, there is
$\{u_n\}_{n\in\N}\subset \H$ with $I(u_n)\rightarrow \nu$ and $I'(u_n)=0$. 
We deduce that $\{u_n\}_{n\in\N}$ is bounded in $\H$.
Up to a sequence, for some $\bar{B}>0$,
$$
\int_{\R^N}|(-\triangle)^{\frac{\alpha}{2}} u_n|^2dx\rightarrow \bar{B}^2.
$$
Let us set $J(u):=J_1(u)$ and $J^\infty(u):=J^\infty_1(u)$, for any $u\in \H$.
From Lemma \ref{Lem:global} there exists $u_0\in \H$ such that $u_n\rightharpoonup u_0$ in $\H$ and $J'(u_0)=0$.
Furthermore, there exist $k\in\N\cup\{0\}$,
nontrivial critical points $w^1,\ldots,w^{k}$ of $J^\infty$
and $k$ sequences of points
$\{y_n^j\}_{n\in\N}\subset\R^N$, $1\leq j\leq k$, such that
\begin{equation}\label{eqn:4.1611-}
	\aligned
	\Big\|u_n-u_0-\sum\limits_{j=1}^{k}w^{j}(\cdot-y_n^j)\Big\|\rightarrow0,\quad \nu+\frac{b\bar{B}^4}{4}=
	J(u_0)+\sum\limits_{j=1}^{k}J^\infty(w^j)
	\endaligned
\end{equation}
and
\begin{equation*}
	\bar{B}^2=\|(-\triangle)^{\frac{\alpha}{2}}
	u_0\|_2^2+\sum\limits_{j=1}^{k}\|(-\triangle)^{\frac{\alpha}{2}}w^j\|_2^2.
\end{equation*}
If $k=0$, we are done. If $k\geq 1$, assume by contradiction that $u_0\not=0$. Then, as in Lemma~\ref{Lem:compact}, 
\begin{equation}\label{eqn:4.1711}
	J(u_0)>\frac{b\bar{B}^2}{4}\int_{\R^N}|(-\triangle)^{\frac{\alpha}{2}} u_0|^2dx,
\end{equation}
for each $j$ there is $t_j\in(0,1]$ such that
$I^\infty(w^j(\cdot/t_j))\ge c_1$, which is strict if
$V(x)\not\equiv V_\infty$, and
$$
J^\infty(w^j)\ge c_1+\frac{b\bar{B}^2}{4}\int_{\R^N}|(-\triangle)^{\frac{\alpha}{2}} w^j|^2dx,
$$
where the inequality is strict if $V(x)\not\equiv V_\infty$. Then
by formulas \eqref{eqn:4.1611-}-\eqref{eqn:4.1711} and $\nu\le c_1$, we get
\begin{equation*}
	c_1+\frac{b\bar{B}^4}{4}\geq 
	\nu+\frac{b\bar{B}^4}{4}=J(u_0)+\sum\limits_{j=1}^{k}J^\infty(w^j)> kc_1+\frac{b\bar{B}^4}{4},
\end{equation*}
a contradiction. Hence $u_0=0$ and $k=1$, in which case
a contradiction follows as in the proof of Lemma~\ref{Lem:compact}.
The proof is complete.
\qed



\begin{thebibliography}{10}

\bibitem{Alves01}
C. O. Alves, F. Corr\^{e}a, {\em On existence of solutions for a class of problem involving a nonlinear operator}, Appl. Nonlinear Anal. {\bf 8} (2001), 43--56.


\bibitem{Ambrosio16} V. Ambrosio, T. Isernia, {\em A multiplicity result for a fractional Kirchhoff equation in $\R^N$ with a general nonlinearity}, preprint.

\bibitem{Autuori15} G. Autuori, A. Fiscella, P. Pucci, {\em Stationary Kirchhoff problems involving a fractional elliptic operator and a critical nonlinearity}, Nonlinear Anal. {\bf 125} (2015), 699--714.


\bibitem{Azzollini11} A. Azzollini,  {\em The elliptic Kirchhoff equation in $\R^N$ perturbed by a local nonlinearity}, Differential Integral Equat. {\bf 25} (2012), 543--554.


\bibitem{baer} A.\ Baernstein, A unified approach to symmetrization. In Partial differential equations of elliptic type (Cortona, 1992), 47--91. Symposia Mathematica {\bf 35} Cambridge University Press, Cambridge, 1994



\bibitem{Berestycki1} H. Berestycki, P. Lions, {\em Nonlinear scalar field equations. \emph{I}. Existence of a ground state}, Arch. Ration. Mech. Anal. {\bf  82} (1983), 313--345.

\bibitem{Brezis83-} H. Brezis, L. Nirenberg, {\em Positive solutions of nonlinear elliptic problems involving critical Sobolev exponent}, Comm. Pure Appl. Math. {\bf 36} (1983), 437--477.


\bibitem{Caffarelli07} L. Caffarelli, L. Silvestre, {\em An extension problem related to the fractional Laplacian}, Comm. PDE {\bf 32} (2007), 1245--1260.


\bibitem{Chipot97} M. Chipot, B. Lovat, {\em Some remarks on nonlocal elliptic and parabolic problems}, Nonlinear Anal. {\bf 30} (1997), 4619--4627.

\bibitem{Chang13} X. Chang, Z. Wang, {\em Ground state of scalar field equations involving a fractional Laplacian with general nonlinearity}, Nonlinearity {\bf 26} (2013), 479--494.

\bibitem{Dipierro13} S. Dipierro, G. Palatucci, E. Valdinoci, {\em Existence and symmetry results for a Schr\"{o}dinger type problem involving
the fractional laplacian}, Matematiche LXVIII, (2013), 201--216.

\bibitem{Felmer12}  P. Felmer, A. Quaas, J. Tan, {\em Positive solutions of nonlinear Schr\"{o}dinger equation with the fractional Laplacian},
Proc. Roy. Soc. Edinburgh Sect. A {\bf 142} (2012), 1237--1262.

\bibitem{Fiscella14} A. Fiscella, E. Valdinoci, {\em A critical Kirchhoff type problem involving a nonlocal operator}, Nonlinear Anal. {\bf 94} (2014), 156--170.

\bibitem{Frank13} R. Frank, E. Lenzmann, {\em Uniqueness and non degeneracy of ground states for $(-\triangle)^sQ+Q-Q^{\alpha+1}=0$ in $\R$}, Acta Math.\ {\bf 210} (2013), 261--318.

\bibitem{He11} X. He, W. Zou, {\em Existence and concentration behavior of positive solutions for a Kirchhoff equation in $\R^3$}, J. Differential Equations {\bf 252} (2012), 1813--1834.

\bibitem{Jeanjean99} L. Jeanjean, {\em On the existence of bounded Palais-Smale sequence and application to a Landesman-Lazer type problem set on $\R^N$}, Proc. Roy. Soc. Edinburgh Sect. A {\bf 129} (1999), 787--809.

\bibitem{Kirchhoff83} G. Kirchhoff,  Mechanik, Teubner, Leipzig, 1883.

\bibitem{Laskin02} N. Laskin,  {\em Fractional Schr\"{o}dinger equation}, Phy. Rev. E {\bf 66}, 05618.

\bibitem{Li13} G. Li, H. Ye, {\em Existence of  positive solutions for nonlinear Kirchhoff type problems in $\R^3$ with critical Sobolev exponent and sign-changing nonlinearities}, Math. Methods Appl. Sci. {\bf37}(16) (2014), 2570--2584.

\bibitem{Li14} G. Li, Y. He, {\em Existence of positive ground state solutions for the nonlinear Kirchhoff type equations in $\R^3$}, J. Differential Equations {\bf 257} (2014), 566--600.

\bibitem{Lions78} J. Lions, {\em On some questions in boundary value problems of mathematical physics}, In: Contemporary Developments in Continuum Mechanics and Partial Differential Equations. Proc. Internat. Sympos. Inst. Mat. Univ. Fed. Rio de Janeiro, (1997) In: North-Holland Math. Stud. {\bf 30} (1978), 284--346.

\bibitem{Liu15} Z. Liu,  S. Guo,  {\em Existence and concentration of positive ground states for a Kirchhoff equation involving critical Sobolev exponent}, Z. Angew. Math. Phys. {\bf 66} (2015), 747--769.

\bibitem{Liu15--} Z. Liu, S. Guo, {\em Existence of positive ground state solutions for Kirchhoff type problems}, Nonlinear Anal. {\bf 120 }(2015), 1--13.

\bibitem{Ma03} T. Ma,  J. Rivera,  {\em Positive solutions for a nonlinear nonlocal elliptic transmission problem}, Appl. Math. Lett. {\bf 16} (2003), 243--248.

\bibitem{Nyamoradi13} N. Nyamoradi, {\em Existence of three solutions for Kirchhoff nonlocal operators of elliptic type}, Math. Commun. {\bf 18} (2013), 489--502.

\bibitem{Perera06} K. Perera, Z. Zhang, {\em Nontrivial solutions of Kirchhoff-type problems via the Yang index}, J. Differential Equations  {\bf 221} (2006), 246--255.

\bibitem{Pucci16-} P. Pucci, S. Saldi, {\em Critical stationary Kirchhoff equations in $\R^N$ involving nonlocal operators}, Rev. Mat. Iberoam {\bf 32}, (2016) 1--22.

\bibitem{Pucci16} P. Pucci, M. Xiang, B. Zhang, {\em Existence and multiplicity of entire solutions for fractional p-Kirchhoff equations}, Adv. Nonlinear Anal. {\bf 5} (2016), 27--55.

\bibitem{Secchi13} S. Secchi, {\em Ground state solutions for nonlinear fractional Schr\"{o}dinger equations in $\R^N$}, J. Math. Phys. {\bf 54} (2013), 031501.

\bibitem{Servadei13--} R. Servadei, E. Valdinoci, {\em The Brezis-Nirenberg result for the fractional Laplacian}, Trans.\ Amer.\ Math. Soc. {\bf 367} (2015), 67--102.


\bibitem{Silvestre06} L. Silvestre, {\em H\"{o}lder estimates for solutions of integro-differential equations like the fractional Laplace}, Indiana Univ.\ J. Math. {\bf 55} (2006), 1155--1174.

\bibitem{symmontrick} M. Squassina, {\em On the Struwe-Jeanjean-Toland monotonicity trick}, Proc. Roy. Soc. Edinburgh Sect. A {\bf 142} (2012), 155--169.


\bibitem{Teng16} K. Teng, {\em Existence of ground state solutions for the nonlinear fractional Schr\"{o}dinger-Poisson system with critical Sobolev exponent},
J. Differential Equations, {\bf 261} (2016) 3061--3106.

\bibitem{jean} J.\ Van Schaftingen, {\em Symmetrization and minimax principles}, Commun. Contemp. Math. {\bf 7}(2005), 463--481.


\bibitem{Wu11} X. Wu, {\em Existence of nontrivial solutions and high energy solutions for Schr\"{o}dinger-Kirchhoff-type equations in $\R^N$}, Nonlinear Anal. RWA.  {\bf 12} (2011), 1278--1287.

\bibitem{Xiang15} M. Xiang, B. Zhang, X. Guo,  {\em Infinitely many solutions for a fractional Kirchhoff type problem via Fountain Theorem}, Nonlinear Anal. {\bf 120} (2015), 299--313.


\bibitem{Zhang06} Z. Zhang, K. Perera, {\em Sign changing solutions of Kirchhoff type problems via invariant sets of descent flow}, J. Math. Anal. Appl. {\bf 317} (2006), 456--463.


\bibitem{Zhang12} J. J. Zhang, W. M. Zou, {\em A Berestycki-Lions theorem revisted}, Commun. Contemp. Math. {\bf 14} (2012), 14 pages.

\bibitem{Zhang14} J. J. Zhang, W. M. Zou, {\em The critical case for a Berestycki-Lions theorem}, Science China Math.\ {\bf 14} (2014), 541--554.

\end{thebibliography}
\end{document}